\documentclass[11pt, oneside]{article}

\usepackage{amsmath, amssymb, amsthm}
\usepackage{geometry}
\usepackage{enumerate}
\usepackage{cases}
\usepackage[font=small,labelfont=bf, width=.85\textwidth]{caption}
\usepackage{graphicx}
\usepackage{pgf,tikz,pgfplots}
\usepackage{tikz-3dplot}
\usepackage[T1]{fontenc}
\usepackage{xcolor}
\DeclareMathOperator{\odd}{odd}
\newcommand{\ve}{\varepsilon}

\newcommand{\pbar}{\overline{\varphi}}
\usepackage[
pdfauthor={},
pdftitle={},
pdfstartview=XYZ,
bookmarks=true,
colorlinks=true,
linkcolor=blue,
urlcolor=blue,
citecolor=blue,
bookmarks=true,
linktocpage=true,
hyperindex=true
]{hyperref}
\newtheorem{theorem}{Theorem}

\newtheorem{lemma}[theorem]{Lemma}

\newtheorem{claim}{Claim}

\newenvironment{proofc}{\begin{proof}[Proof of Claim]}{\end{proof}}

\begin{document}

\title{Total coloring graphs with large minimum degree}
\author{
Owen Henderschedt\footnote{Auburn University, Department of Mathematics and Statistics, Auburn U.S.A.
  Email: {\tt olh0011@auburn.edu}.}    
\qquad
Jessica McDonald\footnote{Auburn University, Department of Mathematics and Statistics, Auburn U.S.A.
  Email: {\tt mcdonald@auburn.edu}.   
Supported in part by Simons Foundation Grant \#845698}
\qquad
Songling Shan\footnote{Auburn University, Department of Mathematics and Statistics, Auburn U.S.A.
  Email: {\tt szs0398@auburn.edu}. Supported in part by NSF  grant 
			DMS-2345869.}
}

\date{}

\maketitle

\begin{abstract}  
We prove that for all $\varepsilon>0$, there exists a positive integer $n_0$ such that if $G$ is a graph on $n\geq n_0$ vertices with $\delta(G)\geq\tfrac{1}{2}(1 + \varepsilon)n$,  then $G$ satisfies the Total Coloring Conjecture, that is, $\chi_T(G)\leq \Delta(G)+2$.
\end{abstract}

\vspace*{.3in}

\section{Introduction}

A \emph{total coloring} of a graph $G$ is an assignment of colors to both the edges and vertices of $G$ so that the  coloring on the edges induces an edge-coloring,  and the coloring on the vertices induces a vertex-coloring, and an incident edge and vertex must also receive different colors. The \emph{total chromatic number} of $G$, denoted $\chi_T(G)$, is the minimum number of colors needed for a total coloring of $G$. In this paper we use the term `graph' to mean simple graph, and will explicitly say `multigraph' when parallel edges are allowed; loops will not appear in this paper.  For terms not defined here, we follow \cite{WestText}.

The \emph{Total Coloring Conjecture}, proposed independently by  Behzad~\cite{Behzad-total-coloring} and Vizing~\cite{Vizing-total-coloring} in the 1960's, asserts that $\chi_T(G) \le \Delta(G)+2$ for every graph $G$, where $\Delta(G)$ is the maximum degree of $G$. Note that since a vertex of maximum degree must receive different colors from all its incident edges, every graph $G$ satisfies $\chi_T(G) \ge \Delta(G)+1$. The conjecture is therefore a natural analog of Vizing's Theorem which states that $\Delta(G) \leq \chi'(G) \leq \Delta(G)+1$, where $\chi'(G)$ is the chromatic index of $G$. Determining each of  $\chi_T(G)$ and $\chi'(G)$  exactly is  known to be NP-hard~\cite{Holyer,NPhTotal}.

There is a large literature towards the Total Coloring Conjecture that we will not survey; the interested reader is referred to~\cite{DMS,  MR4679850, JT}. Relevant to our discussion here is that the conjecture appears to get easier as the maximum degree increases, particularly when the graph is regular. In 1993 Hilton and Hind ~\cite{MR1226136} confirmed the Total Coloring Conjecture for all graphs with $\Delta(G) \ge \frac{3}{4} |V(G)|$. Recently the second and third authors, along with Dalal, lowered this bound to almost $\frac{1}{2} |V(G)|$ for sufficiently large regular graphs~\cite{DMS} . In this paper we extend that result in two directions, by replacing $\Delta(G)$ with minimum degree $\delta(G)$ and by removing the requirement of regularity.

\begin{theorem}\label{main} For all $\varepsilon>0$, there exists $n_0\in\mathbb{N}$ such that if $G$ is a graph on $n\geq n_0$ vertices with $\delta(G)\geq\tfrac{1}{2}(1 + \varepsilon)n$,  then $\chi_T(G)\leq \Delta+2$.
\end{theorem}

Let $G$ be a graph, and let $M$ be a matching in the complement $\overline{G}$ of $G$. Define $G^M$ as the graph obtained from $G$ by first adding one new vertex $x$, and then adding both the edges of $M$ and the edges of $E(x)=\{xu: u\in V(G)\setminus V(M)\}$. We say an edge-coloring $\varphi$ of $G^M$ is \emph{good} if $\varphi$ has at most $\Delta(G)+2$ colors and each edge of $M\cup E(x)$ is colored with a different color under $\varphi$. Note that a good edge-coloring of $G^M$ immediately leads to a total coloring of $G$ with at most $\Delta(G)+2$ colors, because the auxiliary edge-colors can be used on the vertices of $G$. 

Our proof of Theorem \ref{main} has two main parts: first we reduce Theorem \ref{main} to two different cases for $G, M$, and then we show that both of these cases lead to $G^M$ having a good edge-coloring. Section 2 of this paper contains our reduction, which relies on some classic results of Hajnal and Szemer\'{e}di \cite{HS}, Hakimi \cite{H1962}, Vizing and Gupta \cite{Vizing-2-classes, Gupta-67} and Tutte and Berge \cite{TutteTBformula, BergeTBformula} along with some more recent work of Shan \cite{SHAN2022429} and Plantholt and Shan \cite{Plantholt-Shan}.  Section 3 contains a new lemma about edge-coloring certain multigraphs so that specific edges are guaranteed to have different colors. Section 4 uses this lemma to complete the proof of Theorem \ref{main}, namely by showing that $G^M$ has a good edge-coloring in both of the cases we have reduced to.

\section{Reduction}

Let $G$ be an $n$-vertex graph, and define $V_\delta(G)$ to be the set of all minimum-degree vertices in $G$, that is, $V_\delta(G)=\{v\in V(G)| d_G(v)=\delta(G)\}$. For $v\in V(G)$, we say it has \emph{middle-degree} if $\delta(G)<d_G(v)<\Delta(G)$. Finally, for any $\xi>0$ we define
$$U_\xi(G)=\{u\in V(G)| \Delta(G)-d_G(u)\geq \xi n\};$$
we can think of $U_\xi(G)$ as being those vertices in $G$ of sufficiently small degree. 

Consider the following theorem.

\begin{theorem}\label{thm:reduced}
    For all $\varepsilon>0$, there exists $\xi $ with  $0<\xi  \ll \ve $ and $ n_0\in\mathbb{N}$ such $G^M$ has a good edge-colouring, provided $G, M$ satisfy all of the following:
    \begin{itemize}
    \item[(a)]  $G$ is a graph on $n\geq n_0$ vertices with $\delta(G)\geq\tfrac{1}{2}(1 + \varepsilon)n$ and $\Delta(G)<\frac{3}{4}n$;
    \item[(b)] $M$ is a matching in $\overline{G}$, and;
    \item[(c)] $G, M$ lies in one of the following two cases:
     \begin{enumerate}[(1)]
        \item $|M|=\lfloor \frac{n}{2} - 0.1 \ve n \rfloor$ and either $G$ is regular or $|V_\delta(G)|<\xi n$, the quantity $n-|V_\delta(G)|$ is odd, and $G$ has no middle-degree vertices;
  \item $|M|\in\{\lfloor \frac{n}{2} - 0.1 \ve n \rfloor, n-1-\Delta(G) \}$ and $|U_\xi(G)| \ge \xi n$.
    \end{enumerate}
    \end{itemize}
\end{theorem}

We shall prove in this section that Theorem \ref{thm:reduced} implies Theorem \ref{main}. In order to do so, we'll need the following four classic results. 

\begin{theorem}[Hakimi \cite{H1962}]\label{thm:degree-sequence-multigraph} Let $0\leq d_n\leq \ldots \leq d_1$ be integers. Then there exists a multigraph $G$ on vertices $v_1,v_2,\ldots,v_n$ such that $d_{G}(v_i)=d_i$ for all $i\in [1,n]$ if and only if $\sum\limits_{i=1}^{n}d_i$ is even and $\sum\limits_{i=2}^{n}d_i\geq d_1$.
\end{theorem}

\begin{theorem}[Dirac \cite{Dirac}]\label{thm:dirac} If $G$ is an $n$-vertex graph with $n\geq 3$ and $\delta(G)\geq \tfrac{n}{2}$, then $G$ has a Hamilton cycle. 
\end{theorem}

\begin{theorem}[Hajnal and Szemer\'edi~\cite{MR0297607}]\label{them:Hajnal–Szemeredi}
	Let $G$ be a graph and let $k\ge \Delta(G)+1$ be any integer. Then $G$ has an equitable $k$-coloring, that is, a $k$-coloring where the number of vertices in any two distinct color classes differ by at most one.
\end{theorem}

\begin{theorem}[Gupta \cite{Gupta-67} and Vizing \cite{Vizing-2-classes}]\label{thm:chromatic-index}
Every multigraph  $G$ satisfies $\chi'(G) \le \Delta(G)+\mu(G)$. 
\end{theorem}

Let $k\ge 0$ be an integer. A $k$-edge-coloring of a multigraph $G$ is said to be \emph{equalized} if each color class contains either $\lfloor |E(G)|/k \rfloor$ or $\lceil |E(G)|/k \rceil$ edges.  McDiarmid~\cite{MR300623} observed the following result that we will find useful. 
\begin{theorem}\label{lem:equa-edge-coloring}
	Let $G$ be a multigraph with chromatic index $\chi'(G)$. Then for all $k\ge \chi'(G)$, there is an equalized edge-coloring of $G$ with $k$ colors. 
\end{theorem}

We will also need the following results of Shan \cite{SHAN2022429} and Plantholt and Shan \cite{Plantholt-Shan}.

\begin{lemma}[Shan~{\cite[Lemma 3.2]{SHAN2022429}}]\label{lem:partition}
	There exists a positive integer $n_0$ such that for all $n\ge n_0$ the
	following holds. Let $G$ be a graph on $2n$ vertices, and $N=\{x_1,y_1,\ldots, x_t,y_t\}\subseteq V(G)$, where $ 1\le t\le n$ is an integer.  
	Then $V(G)$ can be partitioned into two  parts 
	$A$ and $B$ satisfying the properties below:
	\begin{enumerate}[(i)]
		\item  $|A|=|B|$;
		\item $|A\cap \{x_i,y_i\}|=1$ for each $i\in \{1,\ldots, t\}$;
		\item $| d_A(v)-d_B(v)| \le n^{2/3}$ for each $v\in V(G)$. 
	\end{enumerate}
	Furthermore, one such partition can be constructed in $O(2n^3 \log_2 (2n^3))$-time. 
\end{lemma}

\begin{lemma}[Plantholt and Shan\cite{PS2023}]\label{lem:matching} Let $0<1/n_0\ll \varepsilon <1$ and $G$ be a graph on $n\geq n_0$ vertices such that $\delta(G)\geq (1+\ve)n/2$. Moreover, let $M=\{a_1b_1,\ldots,a_tb_t\}$ be a matching in the complete graph on $V(G)$ of size at most $\ve n/8$. Then there exist vertex-disjoint paths $P_1,P_2,\ldots,P_t$ in $G$ such that $\bigcup_{i=1}^{t} V(P_i)=V(G)$, $P_i$ joins $a_i$ to $b_i$, and these paths can be found in polynomial time.
\end{lemma}

We will prove one new lemma about matchings to use for our reduction, and for this we need the Tutte-Berge Formula.

\begin{theorem}[Tutte \cite{TutteTBformula} and Berge \cite{BergeTBformula}] Let $H$ be an $n$-vertex graph. Then the size of a maximum matching in $H$ is equal to $\tfrac{1}{2}\mathop{\min}_{S\subseteq V(G)}\{n+|S|-\odd(H-S) \}$, where $\odd(H-S)$ denotes the number of  components  of odd order in $H-S$. 
\end{theorem}

Within the proof of our matching lemma, which follows, and indeed for our later work as well, we use $E_G(S,T)$ denote the set of edges of a graph $G$ with one endpoint in some $S\subseteq V(G)$ and one endpoint in $T\subseteq V(G)$. We let $e_G(S,T) = |E_G(S,T)|$. For convenience we may use this same notation when $S$ or $T$ is a subgraph of $G$, instead of writing $V(S), V(T)$.

\begin{lemma}\label{lem: large-matching}
Let $ 0<\xi  \ll 1$ and $n_0\in \mathbb{N}$ such that $\frac{1}{n_0} \ll \xi$.   Let $G$ be a graph on  $n \ge n_0$ vertices  with $\delta(G)\geq\frac{n}{2}$, $\Delta(G) < \frac{3n}{4}$ and  $|U_\xi(G)| < \xi n$.  Then there exists a matching $M$ in $\overline{G}$ with $|M|\geq \frac{n}{2}-1.5\xi n$.
\end{lemma}

\begin{proof}
Let $H= \overline{G}$. Then  $\Delta(H) <  \frac{n}{2}$ and $\delta(H)>\frac{n}{4}-1$.
Let  $S\subseteq V(H)$. We'll show that $\odd(H-S)-|S|< 3\xi n$, from which the Tutte-Berge Formula guarantees a matching in $H$ of size at least $\frac{n}{2}-1.5\xi n$. 

Let $\odd(H-S)=p+q$, with $p$ counting those odd components of order at most $\xi n$, and $q$ counting those odd components of order larger than $\xi n$. Note that $q<\frac{1}{\xi}$. To get an upper bound on $p$, we consider an odd component $D$ with $ 1\le |V(D)| \le \xi n$. We know that $\xi n <\delta(H)$ since $\xi\ll\frac{1}{4}-1$. By counting the total number of neighbors of each vertex in $D$, we have 
\begin{eqnarray*}
    e_H(D,S) &\ge &|V(D)| (\delta(H)-|V(D)|+1)\\
    &\geq &  \delta(H), 
\end{eqnarray*}
 as the function $|V(D)| (\delta(H)-|V(D)|+1)$ is concave down in $|V(D)|$
 and so takes its minimum at the two endpoints of $|V(D)|$. So counting these edges for every such $D$, we get
 $$
  p \cdot \delta(H) \leq e_H(S, H-S).$$
We therefore get that
\begin{equation}\label{eqn:pq}
\odd(H-S)=p+q < \frac{e_H(S, H-S)}{\delta(H)}+\frac{1}{\xi}.
\end{equation}
For any $v\notin U_\xi(G)$, $d_G(v)> \Delta(G)-\xi n$ and hence $d_H(v)< \delta(H)+\xi n$. So we get
\begin{eqnarray*}
    e_H(S,H-S) &  \le    & |S\cap U_\xi(G)|\cdot \Delta(H)+|S\setminus U_\xi(G)|\cdot(\delta(H)+\xi n)  \\
    &=& |S\cap U_\xi(G)|\cdot \Delta(H) +|S|(\delta(H)+\xi n) -|S\cap U_\xi(G)|(\delta(H)+\xi n) \\
    &<& |S\cap U_\xi(G)|\cdot\tfrac{n}{2}+|S|(\delta(H)+\xi n) -|S\cap U_\xi(G)|(\delta(H)+\xi n) \\
    &<& (\xi n)\tfrac{n}{2} + (|S|-\xi n)(\delta(H)+\xi n), 
\end{eqnarray*}
where we get the last inequality above as $|S\cap U_\xi(G)|\cdot\tfrac{n}{2} -|S\cap U_\xi(G)|(\delta(H)+\xi n)$ is a function increasing in $|S\cap U_\xi(G)|$
and $|S\cap U_\xi(G)| < \xi n$. 
Substituting this into (\ref{eqn:pq}) gives
$$ \odd(H-S) \le  \frac{\xi n^2}{2\delta(H)}+|S|-\xi n + \frac{(|S|-\xi n) \xi n}{ \delta(H)}+\frac{1}{\xi}.$$
Using $\delta(H)> \tfrac{n}{4}-1$ and $|S| <\frac{n}{2}-\xi n <2\delta(H)$, we get
\begin{equation}\label{eqn:xin} \odd(H-S) < 2 \xi n +|S|-\xi n + 2\xi n - \frac{\xi^2 n^2}{\delta(H)} +\frac{1}{\xi} =3\xi n +|S| - \frac{\xi^2 n^2}{\delta(H)} +\frac{1}{\xi}.\end{equation}
Note that $n\geq n_0\gg \tfrac{1}{\xi} >\tfrac{1}{\xi^3}$, so from (\ref{eqn:xin}) we get
$$\odd(H-S) < 3\xi n +|S| - \frac{1}{\xi}\left(\frac{n}{\delta(H)}\right) +\frac{1}{\xi}.$$
Since $\delta(H)<n$ this last line tells us that $\odd(G-S)-|S|< 3\xi n$, as desired.
\end{proof}

We can now prove the main result of this section.\\

\begin{theorem} If Theorem \ref{thm:reduced} is true, then Theorem \ref{main} is true.
\end{theorem}

\begin{proof} We have $\varepsilon>0$, and we choose  
$n_0,\xi$  so that 
$n_0\in \mathbb{N}$ is at least the lower bound of $n$ in Theorem~\ref{thm:reduced}, 
and that 
\begin{equation}\label{eqn:parameters}
	0< \frac{1}{n_0 }   \ll  \xi  \ll \ve.    
\end{equation}
Let $G$ be a graph on $n\ge n_0$ vertices with $\delta(G) \ge \frac{1}{2}(1+\varepsilon)n$. By the result of Hilton and Hind~\cite{MR1226136} mentioned in the introduction,  we may assume that $\Delta(G)< \frac{3}{4} n$. Since we are assuming that Theorem \ref{thm:reduced} is true, in order to prove Theorem 1 it suffices to find a matching $M$ in $\overline{G}$ so that $G, M$ satisfy
either (1) or (2) of Theorem \ref{thm:reduced}; we rewrite these conditions here for convenience:
\begin{enumerate}[(1)]
        \item $|M|=\lfloor \frac{n}{2} - 0.1 \ve n \rfloor$ and either $G$ is regular or $|V_\delta(G)|<\xi n$, the quantity $n-|V_\delta(G)|$ is odd, and $G$ has no middle-degree vertices;
  \item $|M|\in\{\lfloor \frac{n}{2} - 0.1 \ve n \rfloor, n-1-\Delta(G) \}$ and $|U_\xi(G)| \ge \xi n$.
    \end{enumerate}
We divide our proof into different cases according to the sizes of $U_\xi(G)$ and $V_\delta(G)$.\\

\noindent \textbf{Case 1: $|U_\xi(G)| \geq \xi n$}.  

Theorem~\ref{them:Hajnal–Szemeredi} says that $G$ has an equitable-coloring using $\Delta(G)+1$ colors. As $\Delta(G)\geq \delta(G)>\frac{n}{2}$, each color class has size 2 or 1. Let $p$ be the total number of color classes of size 2, and $q$ be that of size 1. Then we have $p+q=\Delta(G)+1$ and $2p+q =n$. This implies that $p=n-1-\Delta(G)$. So, there exists a matching $M_0$ in $G$ of size $n-1-\Delta(G)$.  

Since $|U_\xi(G)| \ge \xi n$, then we choose in $\overline{G}$ a matching $M$ of size $\lfloor \frac{n}{2} - 0.1 \ve n \rfloor$ if it exists; otherwise, we take $M=M_0$. Either way, we meet the conditions of (2).\\

\noindent \textbf{Case 2: $|U_\xi(G)| < \xi n$}. 

By the assumption of this case we may apply Lemma \ref{lem: large-matching} to get a matching $M$ in $\overline{G}$ of size  $\lfloor \frac{n}{2} - 0.1 \ve n \rfloor$. If $G$ is regular, then $G, M$ meet the conditions of (1), so we may assume not. Moreover, we may assume that at least one of the following three conditions is \emph{not} met: $|V_\delta(G)|<\xi n$; the quantity $n-|V_\delta(G)|$ is odd, and; $G$ has no middle-degree vertices. We divide our argument into two subcases according to whether the first of these conditions is met or not.\\

\noindent \textbf{Case 2a: $|V_\delta(G)| \geq \xi n$}. 

Let $G'=G$ unless $\Delta(G)$ is odd and $n$ is odd. In this latter case, note that $G$ has a Hamilton cycle by Theorem \ref{thm:dirac} and by taking every second edge in this cycle we can pick a near-perfect matching $F$ of $G$ where the only unsaturated vertex is some $v_1$ with $d_G(v_1)=\delta(G)$. Then let $G'=G-F$, and note that $\Delta':=\Delta'(G)=\Delta(G)-1$ in this case.

Suppose $V(G')=\{v_1, \ldots, v_n\}$ with $d_{G'}(v_1) \le d_{G'}(v_2) \le \ldots \le d_{G'}(v_n)$. Let $d_i=\Delta'-d_{G'}(v_i)$. 
Then we have 
$$\sum_{i=1}^n d_i = \sum_{i=1}^n (\Delta'-d_{G'}(v_i))=n\Delta'-2e(G'),$$
which is an even number since $n\Delta'$ is always even (by our definition of $G'$). Moreover, $|V_\delta(G)| \ge \xi n\gg \tfrac{n}{n_0}\geq 1$, so in particular $d_1=d_2$ and hence $d_1\leq \sum_{i=2}^n d_i$. We can therefore apply Theorem~\ref{thm:degree-sequence-multigraph} to construct a multigraph  $L$ on $\{v_1, \ldots, v_n\}$
with degree sequence $(d_1, \ldots, d_n)$.  

Since $|V_\delta(G)| \geq \xi n$ we know that $|V_\delta(G)|> |U_{\xi}(G)|$, which means that $\Delta(G)-\delta(G) < \xi n$. By Theorem~\ref{thm:chromatic-index}, we have 
 $$\chi'(L)\leq \Delta(L)+\mu(L)\leq 2\Delta(L)\leq 2(\Delta(G)-\delta(G))< 2\xi n.$$
 Setting $k_0= \lceil 2\xi n \rceil$ and applying Theorem \ref{lem:equa-edge-coloring}, we can partition $E(L)$ into a set of $k_0$ matchings each with size at most $$\left\lceil\frac{e(L)}{k_0}\right\rceil\leq \left\lceil\frac{n (\Delta(G)-\delta(G))/2}{k_0}\right\rceil \leq \left\lceil\frac{n(\xi n)/2}{2\xi n}\right\rceil \leq  \left\lceil\frac{n}{4}\right\rceil.$$ 
Let us now partition each of these $k_0$ matchings into matchings of size at most $\tfrac{1}{2}\xi^{1/2}n$. Let $M_1, \ldots, M_k$ be the resulting set of matchings in $L$, with 
$$k\leq  2\xi n \left\lceil \frac{\lceil n/4\rceil}{ \tfrac{1}{2}\xi^{1/2}n} \right\rceil \leq 4\xi n \left( \frac{ n/4}{ \tfrac{1}{2}\xi^{1/2}n} \right) = 2\xi^{1/2} n. $$

Our aim now will be to  define spanning subgraphs $G'_0,G'_1,\ldots, G'_k$ of $G'$ such that
 \begin{enumerate}[(a)]
	 	\item $G'_0:=G'$ and $G'_i:=G'_{i-1}- E(F_i)$ for $i\geq 1$.
	 	\item $F_i$ is a spanning linear forest in $G'_{i-1}$ whose leaves are precisely the vertices in $V(M_i)$.
	 \end{enumerate}
Let $G'_0=G'$ and suppose that for some $i\in [1,k]$, we have already defined $G'_0,G'_1,\ldots,G'_{i-1}$ and $F_1,\ldots,F_{i-1}$. Since $$\Delta(F_1\cup\ldots \cup F_{i-1})\leq 2(i-1)\leq 2(k-1)\leq 4\xi^{1/2} n,$$
it follows that 
\begin{equation}\label{eqn:minGprime}
\delta(G'_{i-1})\geq (1+\ve)n/2-4\xi^{1/2} n \geq (1+0.9\ve)n/2.
\end{equation}
Since $M_i$ has size at most $\xi^{1/2} n/2$,  we can apply Lemma~\ref{lem:matching} to $G'_{i-1}$ and $M_i$ to obtain a spanning linear forest $F_i$ in $G'_{i-1}$ whose leaves are precisely the vertices in $V(M_i)$. Set $G'_i=G'_{i-1}-E(F_i)$.

We claim that $G'_k$  is regular. Consider any vertex $u\in V(G'_k)$. For every $i\in [1,k]$, $d_{F_i}(u)=1$ if $u$ is an endvertex of some edge of $M_i$ and $d_{F_i}(u)=2$ otherwise. Since $M_1,M_2,\ldots,M_k$ partition $E(L)$, $u$ is the endpoint of exactly $d_L(u)$ of these matching edges, so
$$\sum\limits_{i=1}^{k}d_{F_i}(u)=2k-d_{L}(u)=2k-(\Delta'-d_{G'}(u)).$$ 
Thus $$d_{G'_{k}}(u)=d_{G'}(u)-\sum\limits_{i=1}^{k}d_{F_i}(u)=d_{G'}(u)-2k+(\Delta'-d_{G'}(u)) =\Delta'-2k.$$
Since the above holds for any $u\in V(G'_k)$, we see that $G'_k$ is indeed regular.

By (\ref{eqn:minGprime}) and by our initial assumptions,
$$(1+0.9\ve)n/2\leq \Delta(G'_k) \leq \Delta(G)< \tfrac{3}{4}{n}.$$
The matching $M$ we found earlier is in $\overline{G}\subseteq \overline{G'_k}$, so since $G'_k$ is regular, $G'_k, M$ satisfy all the conditions of 
Theorem~\ref{thm:reduced}(1), so $(G'_k)^M$ has a good edge-coloring, which uses $\Delta(G_k')+2$ colors.  Since $\Delta(G_k')=\Delta'-2k$, and each of the $k$ linear forests is edge 2-colorable, this translates to a good edge-coloring of $(G')^M$ which uses $\Delta'+2$ colors. If $G=G'$ this is exactly our desired result. Otherwise, $G'=G-F$ with $F$ a matching and $\Delta'=\Delta(G)-1$. So again we get our desired result by using one additional color on the edges of $F$.\\

\noindent \textbf{Case 2b: $|V_\delta(G)| < \xi n$.}

Consider the graph $G'=G-V_\delta(G)$ (which is nontrivial since $G$ is not regular). Then 
$$\delta(G')> \delta(G)- |V_\delta(G)| > (1+\varepsilon) \tfrac{n}{2} -\xi n \geq \tfrac{n}{2} + n(\tfrac{\varepsilon}{2}-\xi)>\tfrac{n}{2},$$
with the last inequality following from $\xi \ll \varepsilon$.
So $G'$ has a Hamilton cycle by Theorem \ref{thm:dirac}. If $G'$ has even order then taking every second edge in this Hamilton cycle gives a perfect matching $M'$ of $G'$; otherwise we get a near-perfect matching $M'$ of $G'$, where $M'$ unsaturates one vertex that has  degree less than $\Delta(G)$ in $G$ if such a vertex exists.
Let $G_1$ be the graph obtained from $G$ by deleting $M'$.

Note that $G_1$ is an $n$-vertex graph with $\delta(G_1)\ge \frac{1}{2}(1+\varepsilon)n$. We may assume that $\Delta(G_1)=\Delta(G)-1$, unless $G$ has no middle-degree vertices and $|V(G')|$ is odd. Of course, in the latter case $G, M$ meet all the assumptions of Theorem \ref{thm:reduced} (1) and hence we are done. So we may indeed assume that $\Delta(G_1)=\Delta(G)-1$. 

If $|V_\delta(G_1)| < \xi n$, then we can repeat the above process and indeed continue this process of deleting matchings until we either get some $n$-vertex graph $G_t$ where $|V_\delta(G_t)|\geq \xi n$, or we get some $n$-vertex graph $G_i$ where $|V_\delta(G_i)| < \xi n$, $G_i$ has no middle-degree vertices, and $|V(G_i)|$ is odd. Note that in both scenarios, $\delta(G_i), \delta(G_t)\geq \frac{1}{2}(1+\varepsilon)n$, since we only delete matchings from $G'$.

Suppose first that we get the graph $G_i$ described in the above paragraph. Then $G_i, M$ meet all the conditions of Theorem \ref{thm:reduced}(1) (note $M\subseteq \overline{G}\subseteq \overline{G_i}$), so $G_i^M$ has a good coloring, which uses $\Delta(G_i)+2$ colors. But $\Delta(G_i)=\Delta(G)-i$, and so by using $i$ additional colors on the $i$ deleted matchings we get a good coloring of $G^M$. 

We may now assume that we get the graph $G_t$ described in the above paragraph, where $|V_\delta(G_t)|\geq \xi n$. We can get a good coloring of $G_t^M$: this is via Case 1 if $|U_\xi(G_t)|\geq \xi n$, and via Case 2a if not. The good coloring of $G_t^M$ uses $\Delta(G_t)+2$ colors.  Since $\Delta(G_t)=\Delta(G)-t$, we can use $i$ additional colors on the $i$ deleted matchings to a good coloring of $G^M$. 
\end{proof}

\section{An edge-coloring lemma}

Given a graph $G$ and an integer $k\ge 0$, a partial $k$-edge-coloring of $G$ is an edge-coloring of some subgraph $G'\subseteq G$. For any partial $k$-edge-coloring, $\varphi$ of $G$, and any vertex $v\in V(G)$, denote by $\pbar(v)$ the set of colors \emph{missing} at $v$, that is, the set of all colors in $\{1, 2, \ldots, k\}$ that are not used on an edge incident to $v$ under $\varphi$. For a vertex set $X\subseteq V(G)$,  define  $\pbar(X)=\bigcup _{v\in X} \pbar(v)$ to be the set of missing colors of $X$. Consider a pair of distinct colors $\alpha, \gamma \in\{1, \ldots, k\}$, and consider some maximal $(\alpha, \gamma)$-alternating path or cycle $P$ in $G$. Note that by switching the colors $\alpha, \gamma$ on $P$ we get a new partial $k$-edge-coloring $\varphi'$ of $P$ where exactly the same set of edges in $G$ have been colored; we denote this $\varphi'$ by $\varphi/P$.

Let $G$ be a graph and let $\varphi$ be a partial $k$-edge-coloring of $G$.  A \emph{multifan} centered at some vertex $r$ with respect to $\varphi$ is a sequence $F_\varphi(r,s_0: s_p)=(r, e_0, s_0, e_1, s_1, \ldots, e_p, s_p)$ with $p\geq 0$ consisting of  distinct edges $e_0, \ldots, e_p$  with $e_i=rs_i$ for all $i$,  where $e_0$ is left uncolored under $\varphi$,   
and \emph{for every edge $e_i$ with $i\in \{1, \ldots, p\}$,  there exists  $j\in \{0, \ldots, i-1\}$ such that 
		$\varphi(e_i)\in \pbar(s_j)$.} 
A subsequence $(s_0=s_{\ell_0}, s_{\ell_1},s_{\ell_2}, \ldots, s_{\ell_t})$ with the property that $\varphi(e_{\ell_i})= \alpha\in \pbar(s_{\ell_{i-1}})$ for each $i\in \{1,\ldots, t\}$, is called a \emph{linear sequence} of $F_\varphi(r,s_0: s_{p})$. 
Given such a linear sequence, and given $h \in \{1, \ldots, t\}$, we \emph{shift from $s_{\ell_h}$ to $s_0$} by recoloring edge $e_{\ell_{i-1}}$ with $\varphi(e_{\ell_{i}})$ for all $i\in\{1, \ldots, h\}$. 
Note that the resulting coloring  is a partial $k$-edge-coloring where $e_0$ is colored and $e_{\ell_h}$ is not. Note that these linear subsequences and shifts have been implicitly used in many papers (see e.g.~\cite{MR4694336}); the definition of multifan first appeared in the book of Stiebitz et al.~\cite{StiebSTF-Book} as a generalization of classic work of Vizing~\cite{Vizing-2-classes}.

The proof of the following lemma is very similar to that of Lemmas 3 and 4  in~\cite{DMS}. Here we use our notation of $E_G(S, T)$ as before, except when $S=\{u\}$ and $T = \{v\}$, we write $E_G(u,v)$. Note that in this lemma we are dealing with multigraphs so $e_G(u,v)$ may be larger than $1$. We define the \emph{multiplicity of $v$} as the quantity, $\mu_G(v) = \max\{e_G(u,v): u\in V(G)\}$ and with the \emph{multiplicity of $G$} being the quantity $\mu(G) = \max\{\mu_G(v): v\in V(G)\}$.

\begin{lemma}\label{lemma:matching-extension}
Let $G$ be a multigraph with $\mu(G) \le 2$,  and let $J_0\subseteq E(G)$ be composed of a matching and the edges of a star centered at a vertex $x$, with the star comprising of all of $x$'s incident edges in $G$ (none of which are from the matching). Moreover, suppose that if $u\neq x$, then $e_G(u, v)=2$ for at most one vertex $v\in V(G)$. Let $J\subseteq E(G)$  with $J_0\subseteq J$ be such that $G-J$ is simple, and suppose that   $G[J]$ has a $k_0$-edge-coloring $\varphi_0$ for some $ \Delta(G)\ge k_0 \ge |J_0|$ for which each edge in $J_0$ receives a different color. Then the following statements hold:
    \begin{enumerate}[(a)]
        \item For any $k\ge \Delta(G)+4$, the edge-coloring   $\varphi_0$ can be  modified  to a $k$-edge-coloring of $G$ such that each edge of $J_0$ receives a different color.  
        \item Suppose that   $G-x$ is bipartite and simple, and $(X,Y)$ is  a bipartition of $G-x$.  If  $\Delta(G[J]-x) \le 1$ and all the vertices of $X$ have degree less than $k$ in $G$, for some $k\geq \Delta(G)$,  then $\varphi_0$ can be  modified  to a $k$-edge-coloring of $G$ such that each edge of $J_0$ receives a different color. 
    \end{enumerate}
\end{lemma}

\proof 
Let $\Delta=\Delta(G)$ and $\varphi$ be a modification of $\varphi_0$ using $k$ colors. Note, by modification we never uncolor edges in $J$, but possibly changing its colors. Moreover, we choose $\varphi$ with the following properties:
\begin{enumerate}
\item[(P1)] every edge in $J_0$ receives a different color, and;
\item[(P2)] subject to (P1), $\varphi$ assigns colors to as many edges as possible.
\end{enumerate}
If every edge in $G$ is colored by $\varphi$ then we are done, so we may assume there is some edge $e_0$ that is uncolored by $\varphi$. Since $e_0$ is uncolored, $e_0\notin J$, so by assumption $e_0$ has endpoints $u,v$ for some $u, v\in V(G)\setminus\{x\}$. In Case (b) we may further assume that $v\in X$. Let $F$ be a maximal multifan centered at $u$. Note that since $u\neq x$, there is at most one vertex $w$ with $e_G(u,w)=2$, and $u$ is incident with at most one edge from $J_0$.  In case (a) we know $k\geq \Delta+4$, so $|\pbar(v)|\geq 5$; in case (b) we know that $k\geq \Delta$ and $v\in X$ so $|\pbar(v)|\geq 2$. In either case, 
there exist  linear sequences of $F$ containing no edge of $J$. Let $F_1\subseteq F$ be a multifan induced by all the linear sequences of $F$ that don't contain edges of $J$; say  $F_1=(u, uw_0, w_0,  uw_1, w_1, \ldots,  uw_t, w_t)$,  where $w_0=v$.

\begin{claim}\label{uwi} Suppose $F_2\subseteq F_1$ corresponds to the union of any number of linear sequences of $F$. Then $\pbar(u)\cap \pbar(w_i) =\emptyset$ for any $w_i\in V(F_2)$.
\end{claim}

\proof[Proof of claim] Suppose for a contradiction that the claim is false. If $i=0$, then $w_i=v$ and by coloring $uv$ any common missing color, we immediately contradict (P2). Thus, we may assume $i\geq 1$. To this end, suppose that $\pbar(u)\cap \pbar(w_i) \ne \emptyset$ for some $i$. Then there exists a linear sequence $(u, w_0, w_{i_1}, \ldots, w_{i_s})$ of $F_2$ for which $w_{i_s}=w_i$. Shift $(u, w_0, w_{i_1}, \ldots, w_{i_s})$ 
from $w_{i_s}$ to $w_0$ and then color the edge $uw_i$ with any color in $\pbar(u)\cap \pbar(w_i)$. The resulting partial $k$-edge-coloring  $\varphi'$ of $G$ still maintains the property that every edge in $J_0$  colored differently, since we excluded the edges of $J_0$ from $F_1$ (and $F_2$). Hence $\varphi'$ contradicts (P2).
\qed

In statement (b), as $G-x$ is bipartite,  for any  distinct $\alpha, \beta  \in [1,k]$ with $\alpha \in \pbar(u), \beta \in \pbar(w_i)$
for some $i\in [0,t]$,  if there is an  $(\alpha,\beta)$-alternating path joining $u, w_i$, then it creates an odd cycle with $uw_i$, and hence we know  that the path must contain $x$ (as an internal vertex). Moreover, this means that the path contains two edges from $J_0$.
Similarly, if $\beta \in \pbar(w_i)\cap \pbar(w_j)$ for 
distinct $i,j\in [1,t]$, and there is an $(\alpha,\beta)$-alternating path joining $w_i$ and $w_j$, then this creates an odd cycle with $uw_i, uw_j$, and hence we know that the path must contain $x$ (as an internal vertex). Again, we also get that such a path contains two edges from $J_0$.

\begin{claim}\label{gamma} For any color $\gamma \in [1,k]$,  we have the following statements hold. 
\begin{enumerate}[(1)]
    \item In statement (a),  there are at most  $3$  vertices from $\{w_0, \ldots, w_t\}$ that are all missing $\gamma$ under $\varphi$.
    \item In statement (b),  there is at  most  $1$  vertex from $\{w_0, \ldots, w_t\}$ that is  missing $\gamma$ under $\varphi$.
\end{enumerate}
\end{claim} 
	
\begin{proofc}  
(1) Suppose for a contradiction that such a color $\gamma$ exists.  Let  $y_1, y_2, y_3,  y_{4}$ be $4$  distinct vertices of $\{w_0, \ldots, w_t\}$ all of which are missing $\gamma$ under $\varphi$. Suppose, without loss of generality, that $y_1$ is the first vertex 
in the order $w_0, w_1, \ldots, w_t$ for which $\gamma\in \pbar(y_1)$.  In particular, this guarantees that $F_1(u, w_0:y_1)$ does not contain any edge that is colored with $\gamma$. 
Let $\alpha\in \pbar(u)$; we know that $\alpha\ne \gamma$ by Claim \ref{uwi}.

Consider the set $\mathcal{P}$ of maximal $(\alpha,\gamma)$-alternating
path containing at least one vertex from $\{u,y_1,...,y_4\}$. Since $\alpha\in \pbar(u)$ and $\gamma\in \cap_{i=1}^4\pbar(y_i)$, each component of $\mathcal{P}$ is a path where every vertex of $\{u,y_1,...,y_4\}$ is an endpoint of a path. Thus, there are at least $3$ paths in $\mathcal{P}$. By (P1), there is at most one edge of $J_0$ colored $\alpha$ and at most one edge of $J_0$ colored $\gamma$. Hence there must exist a path $P\in \mathcal{P}$ which does not contain any edges of $J_0$. We let $\varphi' = \varphi/P$ and consider two different cases depending on whether $u$ is an endpoint of $P$.

\noindent{\bf Case 1.1}:  \emph{The vertex $u$ is an endpoint of $P$.}

Assume first that $y_1$ is the other endpoint of $P$. Then in $\varphi'$, $u$ is missing $\gamma$, so in particular no $\gamma$-edges appear in $F_1$ (if there were any before, they must have been recolored to $\alpha$). However since these edges all occur after $y_1$ in the order of  $F_1$, and $y_1$ is missing $\alpha$ in $\varphi'$, $F_1$ is still a multifan in $G$ under $\varphi'$. However now Claim~\ref{uwi} is not satisfied for $\varphi'$, contradicting  our choice of $\varphi$.

We may now assume that $y_1$ is not the other endpoint of $P$. Here, $F_1(u, w_0:y_1)$ is a multifan with respect to $\varphi'$ containing no edges of $J_0$. But now the color $\gamma$ is missing at both $u$ and $y_1$ under $\varphi'$, so Claim~\ref{uwi} is not satisfied for $\varphi'$, contradicting our choice of $\varphi$.

\noindent{\bf Case 1.2}:  \emph{The vertex $u$ is not an endpoint of $P$.}   

In this case we know that $P$ does not contain any edges of $F_1$. 

First assume that $y_1$ is not an endpoint of $P$. Then $F_1$ is still a multifan with respect to $\varphi'$ containing no edges of $J_0$. However the color  $\alpha$ is missing at both $u$ and at least one other vertex from $\{y_2, y_3, y_4\}$ under $\varphi'$. This means that Claim~\ref{uwi} is not satisfied for $\varphi'$, contradicting our choice of $\varphi$.

We may now assume that $y_1$ is an endpoint of $P$. Now, $F_1(u, w_0:y_1)$ is a multifan with respect to $\varphi'$ containing no edges of $J_0$. But then $\alpha$ is missing at both $u$ and $y_1$ under $\varphi'$, so Claim~\ref{uwi} is not satisfied for $\varphi'$, contradicting our choice of $\varphi$.  This completes our proof of (1).

\noindent (2) Suppose for a contradiction that $y_1$ and $y_2$ are two distinct vertices of $\{w_0, \ldots, w_t\}$, both missing $\gamma$ under $\varphi$. Suppose, without loss of generality, that $y_1$ is the first vertex 
in the order $w_0, w_1, \ldots, w_t$ for which $\gamma\in \pbar(y_1)$. In particular, this implies that $F_1(u, w_0:y_1)$ does not contain any edge that is colored with $\gamma$.  

Let $\alpha\in \pbar(u)$; we know that $\alpha\ne \gamma$ by Claim~\ref{uwi}. Consider the set of  maximal alternating $(\alpha, \gamma)$-paths staring at $u, y_1$ or $y_2$. If these three paths are all disjoint, 
then one of them, say $P$,  does not contain any edge of $J_0$ (since there are at most two edges of $J_0$ with the colors $\alpha, \gamma$). On the other hand, suppose that two of the three paths are in fact the same path $P^*$. By our discussion prior to this claim, $P^*$ must contain $x$ as an internal vertex, and hence there are two $J_0$-edges colored $\alpha, \gamma$ on the path $P^*$. But then again, among the three paths, we can find one, say $P$, that does not contain any edges of $J_0$.

We let $\varphi'=\varphi/P$ and consider two different cases according to whether or not $u$ is an endpoint of $P$. Note that in either case, we have ensured that $P$ doesn't contain any edge of $J_0$, does not contain $x$, and only one of its endpoints is from $F_1$.

\noindent{\bf Case 2.1}:  \emph{The vertex $u$ is an endpoint of $P$.}

Since the other endpoint of $P$ is outside $F_1$, $F_1(u, w_0:y_1)$ is a multifan with respect to $\varphi'$ containing no edges of $J_0$. But now the color $\gamma$ is missing at both $u$ and $y_1$ under $\varphi'$, so Claim~\ref{uwi} is not satisfied for $\varphi'$, leading to a contradiction in our choice of $\varphi$.

\noindent{\bf Case 2.2}:  \emph{The vertex $u$ is not an endpoint of $P$.}   
In this case we know that $P$ does not contain any edges of $F_1$, and our proof is identical to that in Case 1.2 above. 
\qed

\vspace*{.2in}

We can now complete the proof of our lemma. For statement (a), Claims~\ref{uwi} and~\ref{gamma} tell us that
\begin{equation}\label{VF1a}
	|\pbar(V(F_1))|  \ge  |\pbar(u)|+ \tfrac{1}{3}\sum_{i=0}^t|\pbar(w_i)|.
\end{equation}
If we let $G'$ be the graph induced by all the colored edges in $G$, then $|\pbar(z)|\geq k-d_{G'}(z)$ for all $z\in V(G)$. Since $d_{G'}(z)\leq \Delta$, we know that $|\pbar(z)|\geq 4$ for any $z$. Since the edge $uw_0$ is uncolored, we further know that $|\pbar(w_0)|\geq 5$ (and so $t\ge 3$). So from (\ref{VF1a}), we get
\begin{equation}\label{pbarLbound}
|\pbar(V(F_1))|  \ge  (k-d_{G'}(u))+ \tfrac{1}{3}(5)+ \tfrac{1}{3}(4t)>k+2+t-d_{G'}(u).
\end{equation}
On the other hand, we know that there are exactly $t$ colored edges in $F_1$,  so there are exactly $d_{G'}(u)-t$ colored edges incident to $u$ that are not included in $F_1$. At most two of these excluded colored edges could be in $J$ (since $u\neq x$) and by the assumption of the lemma, while all the rest,     under $\varphi$,  must be colored by a color not in $\pbar(\{w_0, \ldots, w_t\})$, by (P2) and by definition of $F_1$. Of course, by virtue of being incident to $u$, none of these excluded colors are in $\pbar(u)$ either. Hence we have found a set of at least $d_{G'}(u)-t-2$ colors from $[1,k]$ that are not in $\pbar(V(F_1)$. So
$$|\pbar(V(F_1))| \leq k-(d_{G'}(u)-t-2)=k+2+t-d_{G'}(u),$$
contradicting (\ref{pbarLbound}).

For statement (b),
Claims~\ref{uwi} and~\ref{gamma} tell us that
\begin{equation}\label{VF1}
	|\pbar(V(F_1))|  \ge  |\pbar(u)|+ \sum_{i=0}^t|\pbar(w_i)|. 
\end{equation}
If we let $G'$ be the graph induced by all the colored edges in $G$, then $\pbar(z)\geq k-d_{G'}(z)$ for all $z\in V(G)$. 
We further argued that $|\pbar(w_0)|\geq 2$ earlier in this proof. Since $w_i\in X$ by the assumption that $u\in Y$, we know that $d_{G'}(w_i) \le d_G(w_i) \le k-1$. Thus $|\pbar(w_i)|\geq 1$ for all other $i$. So from (\ref{VF1}), we get
\begin{equation}\label{pbarLbound1}
	|\pbar(V(F_1))|  \ge k-d_{G'}(u)+ 2+ t.
\end{equation}

On the other hand, we know that there are exactly $t$ colored edges in $F_1$,  so there are exactly $d_{G'}(u)-t$ colored edges incident to $u$ that are not included in $F_1$. At most one of these excluded colored edges could be in $J_0$ (since $u\neq x$ and $\Delta(G[J]-x) \le 1$), while all the rest must have colors that are not in $\pbar(\{w_0, \ldots, w_t\})$, by (P2) and by definition of $F_1$. Of course, by virtue of being incident to $u$, none of these excluded colors are in $\pbar(u)$ either. Hence we have found a set of at least $d_{G'}(u)-t-1$ colors from $[1,k]$ that are not in $\pbar(V(F_1))$. So we get 
$$|\pbar(V(F_1))| \leq k-(d_{G'}(u)-t-1)< k-d_{G'}(u)+ 2+ t,$$
contradicting~\eqref{pbarLbound1}. 
\end{proofc}

\section{Proof of Theorem 2}\label{sec:Theorem1}

In this section we finally prove Theorem \ref{thm:reduced}. In our proof use the following two theorems about bipartite graphs.

\begin{lemma}[Plantholt and Shan \cite{PS2023}]\label{lem:matching-in-bipartite}
	Let $G[X,Y]$ be bipartite graph with $|X|=|Y|=n$. Suppose $\delta(G)=t$ 
for some $t\in [1,n]$, 
	and except at most $t$ vertices all other vertices of $G$ have degree 
	at least $n/2$ in $G$. Then 
	$G$ has a perfect matching. 
\end{lemma}

\begin{theorem}[K\"{o}nig \cite{MR1511872}]\label{thm:konig}
	Every bipartite multigraph $G$ satisfies $\chi'(G)=\Delta(G)$. 
\end{theorem}

We'll additionally use the following two theorems about edge-colorings in multigraphs. The first of these is a simple folklore result (see eg.~\cite[Lemma 2.1]{MR2028248}), and the second is from \cite{DMS}.

\begin{lemma}[Parity Lemma]\label{lem:parity}
	Let $G$ be a multigraph and $\varphi $  be an edge-$k$-coloring  of $G$ for some integer $k\ge \Delta(G)$. 
	Then 
	$|\pbar^{-1}(i)| \equiv |V(G)| \pmod{2}$ for every color $i\in [1,k]$. 
\end{lemma}

\begin{lemma}[Dalal, McDonald, Shan \cite{DMS}]\label{lem:equitable-coloring-precolored-edges}
	Let $G$ be   a multigraph  and $F\subseteq E(G)$. 
	Suppose that for some integer $k\ge \Delta(G)$, 
	$G$ has a $k$-edge-coloring such that all the edges in $F$ receive distinct colors. Then $G$ has 
	a $k$-edge-coloring $\varphi$ such that for any  distinct $i, j\in [1,k]$ it holds that $$\left ||\pbar^{-1}(i)| -|\pbar^{-1}(j)|\right | \le 5,$$ and all edges in $F$ receive distinct colors.
\end{lemma}

In fact, Theorem \ref{lem:equitable-coloring-precolored-edges} was stated only for simple graphs in~\cite{DMS}, but the same proof given there works for multigraphs. 

We can now prove Theorem 2, which we restate here for convenience.

\setcounter{theorem}{1}
\begin{theorem}
    For all $\varepsilon>0$, there exists $\xi $ with  $0<\xi  \ll \ve $ and $ n_0\in\mathbb{N}$ such $G^M$ has a good edge-colouring, provided $G, M$ satisfy all of the following:
    \begin{itemize}
    \item[(a)]  $G$ is a graph on $n\geq n_0$ vertices with $\delta(G)\geq\tfrac{1}{2}(1 + \varepsilon)n$ and $\Delta(G)<\frac{3}{4}n$;
    \item[(b)] $M$ is a matching in $\overline{G}$, and;
    \item[(c)] $G, M$ lies in one of the following two cases:
     \begin{enumerate}[(1)]
        \item $|M|=\lfloor \frac{n}{2} - 0.1 \ve n \rfloor$ and either $G$ is regular or $|V_\delta(G)|<\xi n$, the quantity $n-|V_\delta(G)|$ is odd, and $G$ has no middle-degree vertices;
  \item $|M|\in\{\lfloor \frac{n}{2} - 0.1 \ve n \rfloor, n-1-\Delta(G) \}$ and $|U_\xi(G)| \ge \xi n$.
    \end{enumerate}
    \end{itemize}
\end{theorem}

\proof Let   
$n_0 \in \mathbb{N}$ and $\xi$  be chosen such that 
$n_0$ is  at least the lower bound of $n$ in Lemma~\ref{lem:partition}
and such that 
\begin{equation}\label{eqn:parameters1}
	0< \frac{1}{n_0 }   \ll  \xi  \ll \ve.    
\end{equation} We have different assumptions for $G, M$ according to (1) and (2) above, and we refer to these two different situations as \emph{case 1} and \emph{case 2}, respectively.

Throughout the proof, we let $\Delta=\Delta(G)$, $\delta =\delta(G)$, $U= U_\xi(G)$, and we let $V_\Delta$ and $V_\delta$ the set of 
maximum degree and minimum degree vertices of $G$, respectively. Note that in Case 1, we have $V(G)=V_\Delta \cup V_\delta$.
  
Consider $G^M$, with $x\in V(G^M) \setminus V(G)$. We first define a subgraph $G_1$ of $G^M$ on $m$ vertices as follows:
\begin{numcases}{}
    G_1=G^M- x, \quad m=n & if $n$ is even;  \nonumber \\
    G_1 = G^M, \quad m=n+1 & if $n$ is odd. \nonumber
\end{numcases}
Let $V(G_1)=\{v_1, \ldots, v_m\}$ and assume that $d_{G_1}(v_1) \le \ldots \le d_{G_1}(v_m)$. Note that if $n$ is odd then $v_1=x$ (since $|E(x)|<n/2$). If $n$ is even then we have  $v_1\in V_{\delta}$. 

The remainder of our proof proceeds in two phases: the construction phase and the edge-coloring phase. The construction phase builds a multigraph $Q_{AB}$ in three steps: (1) partitioning $V(G_1)$ into sets $A, B$ using Lemma \ref{lem:partition}; (2)  adding edges to one side of this bipartition in $G_1$ to make a multigraph $Q$ that is somehow easier to deal with and; (3) defining a multigraph $Q_{AB}$ (in most cases a subgraph of $Q$), whose edges we want to color first. The coloring phase proceeds in four steps: (1) edge-coloring $Q_{AB}$ using Lemma \ref{lemma:matching-extension}(a); (2) modifying these color classes so that each is a perfect matching of $Q$ (in case 1) or saturates $V(Q)\setminus U$ and a small subset of vertices of $U$ (in case 2); (3) using new colors on the edges that got uncolored in the previous step, and extending these new color classes in $Q$, and; (4) using Lemma \ref{lemma:matching-extension}(b) to color the remaining edges of $Q$ and get our desired good coloring of $G^M$.  

We proceed now with this work.

\noindent \underline{{\sc Construction Phase.}}

\noindent \textbf{ Step 1: Partition $V(G_1)$ into sets $A$ and $B$.}

We have ensured that $G_1$ is a graph with an even number $m$ of vertices, and we will now pair them into \emph{partners} and label them as $x_i, y_i$ for each $i\in [1,m/2]$.  
To accomplish this, we do a preliminary round of pairing, then we pair up the two endpoints of any edge in $M$ whose endpoints are both still unpaired, and then we pair up the remaining vertices of $G_1$ arbitrarily. If we are in case 2 the preliminary round consists of pairing up $\lfloor \xi n/2\rfloor$ pairs of vertices in $U$. If we are in case 1 and $G$ is regular then we do nothing for the preliminary round. If we are in case 1 and $G$ is not regular, the preliminary round consists of pairing up all the vertices in $V_\delta\setminus \{v_1\}$ (there are an even number of them, since $n-|V_\delta|$ is odd and $v_1\in V_{\delta}$ iff $n$ is even).  Note that in this last situation, since there are no middle-degree vertices, we know that $v_1$ is paired with a vertex from $V_{\Delta}$. In all situations, there are at most $\xi n$ edges of $M$ whose two endpoints are not partnered with each other.

We now apply Lemma~\ref{lem:partition} to $G_1$  and $N=\{x_1,y_1, \ldots, x_{m/2}, y_{m/2}\}$; note this works because $m$ is even. The lemma provides a partition $\{A,B\}$ of $V(G_1)$ satisfying the following properties. Note for a set of vertices $X\subseteq V(G_1)$ we use $N_{G_1}^X(v)$ to denote the vertices in $N_{G_1}(v)\cap X$.
\begin{enumerate}[(P1)]
	\item  $|A|=|B|$;
	\item $|A\cap \{x_i,y_i\}|=1$ for each $i\in [1, \frac{m}{2}]$;
	\item $\big| |N_{G_1}^A(w)|-|N_{G_1}^B(w)|\big| \le (\tfrac{m}{2})^{2/3}$  for each $w\in V(G_1)$. 
\end{enumerate}
By renaming the partition $\{A,B\}$ if necessary, we may assume that $v_1\in B$.

Before moving on to the next step of our construction phase, we pause to prove one claim that will be useful later on.

\begin{claim}\label{claim:Case1-n-odd-e(A)>e(B)}
    Suppose we are in case 1. If $G$ is regular and $n$ is even, then  $e(G_1[A])=e(G_1[B])$. Otherwise,
    \begin{equation}
       \tfrac{1}{2}(\Delta -d_{G_1}(v_1)-2\xi n)  \le e(G_1[A]) -e(G_1[B]) \le \tfrac{1}{2}(\Delta+1+2\xi n -d_{G_1}(v_1)). 
    \end{equation}
\end{claim}

\begin{proofc} If $G$ is regular and $n$ is even then $G_1$ is $G$ with a matching added, so $e(G_1[A])=e(G_1[B])$. So suppose we are not in this case.

Suppose first that $n$ is odd. Then $v_1=x$, and all the vertices of $V_{\delta}$ are distributed evenly between $A$ and $B$. By assumption of the claim, there are no middle-degree vertices in $G$. We know that $v_1\in B$, and $v_1$ is paired with a vertex in $V_\Delta$, which has degree $\Delta+1$ in $G_1$. So we get
 $$
 e(G_1[A]) -e(G_1[B]) =\tfrac{1}{2}\left(\sum_{v\in A}d_{G_1}(v)- \sum_{v\in B}d_{G_1}(v)\right)=\tfrac{1}{2}(\Delta+1-d_{G_1}(v_1)).
 $$

Now suppose that $n$ is even. Then $v_1 \in V_{\delta}\cap B$ is paired with a vertex in $V_{\Delta}\cap A$, all the vertices of $V_{\delta}\setminus \{v_1\}$ are paired up and distributed evenly between $A$ and $B$, and there are no middle-degree vertices in $G$ (since $G$ is not regular). By the construction of $G_1$, we know that $d_{G_1}(v)=d_G(v)+1$ if $v\in V(M)$ and $d_{G_1}(v)=d_G(v)$ if $v\in V(G)\setminus V(M)$.  However, at most 
$2(|V_{\delta}|-1)$ vertices in $V(M)$ are not partnered with the other end of their $M$-edge (the vertices from $V_\delta\setminus \{v_1\}$, and their partners). 

So we get
$$
    \sum_{v\in A} d_{G_1}(v) -\sum_{v\in B} d_{G_1}(v) \le \Delta+1-d_{G_1}(v_1)+ 2(|V_{\delta}|-1),  
$$
 and 
 $$
    \sum_{v\in A} d_{G_1}(v) -\sum_{v\in B} d_{G_1}(v) \ge  \Delta-d_{G_1}(v_1)-2(|V_{\delta}|-1)).  
$$
Since $|V_\delta|< \xi n$ in case 1, these give the desired bounds on $e(G_1[A]) -e(G_1[B])$. 
\end{proofc}

 \vspace*{.1in}

\noindent \textbf{Step 2: Construction of $Q$.}

Our goal in this step is to add edges to one side of the bipartition $(A, B)$ in $G_1$ so that the resulting multigraph $Q$ is somehow easier to deal with. In case 1 this means that we'll ensure $e(Q[A]) = e(Q[B])$ (which will be important in coloring step 3). In case 2 this means we'll ensure that $d_Q(x)$ is not too small when $n$ is odd, i.e. when $x\in V(G_1)$, which will be important in coloring step 2.  In case 2 when $n$ is even, our $G_1$ is already fine as it is, and we set $Q=G_1$.

Suppose first that we are in case 2 with $n$ odd. Then $v_1=x$. We let $B_0\subseteq B$
with $|B_0| = \lceil \frac{1}{2}(\Delta -d_{G_1}(v_1))  \rceil$
and let $Q$ 
be obtained from $G_1$ by adding all edges in $E(v_1):=\{v_1v: v\in B_0\}$.  Note that since $|B|=\tfrac{n+1}{2}$ and $\tfrac{1}{2}\Delta<\tfrac{3n}{8}$ this selection of $B_0$ is possible. 

Suppose now that we are in case 1. We will define $Q$ so that $e(Q[A])=e(Q[B])$. If $n$ is even and $G$ is regular, then by Claim \ref{claim:Case1-n-odd-e(A)>e(B)}, $e(G_1[A])=e(G_1[B])$ so we set $G_1=Q$. Suppose now that either $n$ is odd or $n$ is even but $G$ is not regular. By Claim \ref{claim:Case1-n-odd-e(A)>e(B)},
\begin{equation}\label{eqn:AminusB} \tfrac{1}{2}(\Delta -d_{G_1}(v_1)-2\xi n)  \le e(G_1[A]) -e(G_1[B]) \le \tfrac{1}{2}(\Delta+1+2\xi n -d_{G_1}(v_1)). \end{equation}

If $n$ is odd, we have $v_1=x$ and $d_{G_1}(x)=n-2|M|\geq 2\xi n$, so by (11), we get  $e(G_1[A]) -e(G_1[B])<\tfrac{1}{2}(\Delta+1)<\tfrac{1}{2}(\tfrac{3n}{4}+1)<\tfrac{n}{2}$. Hence we can choose
$B_0\subseteq B$ with $|B_0| = e(G_1[A]) -e(G_1[B]) $ and let $Q$ be obtained from $G_1$ by adding all edges of  $E(v_1):=\{v_1v: v\in B_0\}$.

Suppose now that $n$ is even and $G$ is not regular. If $d_{G_1}(v_1) <\Delta +1  -2\xi n$ then (\ref{eqn:AminusB}) tells us that $e(G_1[A]) -e(G_1[B])>0$. In this case, we let $B_0\subseteq B$ with $|B_0| = e(G_1[A]) -e(G_1[B]) $
and let $Q$ be obtained from $G_1$ by adding all edges of  $E(v_1):=\{v_1v: v\in B_0\}$. On the other hand, if $d_{G_1}(v_1)  \ge \Delta +1 -2\xi n$, then (\ref{eqn:AminusB}) says that $e(G_1[A]) -e(G_1[B]) \le \frac{1}{2}(2\xi n+2\xi n)=2\xi n<\tfrac{n}{4}$. In fact it may be the case that $e(G_1[A])>e(G_1[B])$ here (depending on how many edges of $M$ ended up on one side of the bipartition); whichever of $e(G_1[A]),e(G_1[B])$ is smaller we choose
$|e(G_1[A]) -e(G_1[B])|$ pairs of distinct vertices on that side, and add an edge between each pair of vertices.
We still denote by $B_0$ the set of these chosen vertices. 
As we will see, finding a good coloring in this case is independent of which part contains $v_1$. Thus, we may assume that $e(G_1[B]) \leq e(G_1[A])$; that is, we add the matching to $B$.

In the following step, it will be important to keep track of the degree of $v_1$ during the construction of $Q$, we let $q_1$ be the number of edges added in the construction of $Q$ that are incident to $v_1$, in $B$. The following table gives the values for $q_1$.\\

\begin{table}[h!]
    \renewcommand{\arraystretch}{1.3} 
    \centering
    \begin{tabular}{|l|l|}
    \hline
    Case 1, $n$ even, $G$ regular & $q_1 = 0$ \\
     \hline
        Case 1, $n$ even, $G$ not regular, $d_{G_1}(v_1)<\Delta+1-2\xi n$ & $q_1 = e(G_1[A])-e(G_1[B])$ \\
        \hline
        Case 1, $n$ even, $G$ not regular, $d_{G_1}(v_1)\geq \Delta+1-2\xi n$ & $q_1 \leq 1$ \\
        \hline
        Case 1, $n$ odd & $q_1 = e(G_1[A])-e(G_1[B])$\\
        \hline
        Case 2, $n$ even & $q_1 = 0$\\
        \hline
        Case $2$, $n$ odd & $q_1 = \lceil \frac{1}{2}(\Delta-d_{G_1}(v_1))\rceil$ \\
        \hline
    \end{tabular}
    \caption{The number of edges, $q_1$, added in the construction of $Q$ which are incident to $v_1$ in $B$. Note when $q_1 = e(G_1[A])-e(G_1[B])$, (\ref{eqn:AminusB}) gives bounds on $q_1$.}
    \label{q_1 values}
\end{table}

\medskip 

\noindent \textbf{Step 3: Constructing the multigraph $Q_{AB}$ from $Q$.}

Let 
\begin{numcases}{k=}
    \left\lceil \tfrac{1}{2}\Delta+1.1 \xi n  \right\rceil +4 & in Case 1, \nonumber  \\
     & \label{eqn:definition-k} \\ 
    \left\lceil \tfrac{1}{2}\Delta+ n^{2/3}  \right\rceil+4 & in Case 2.   \nonumber 
\end{numcases}
We will choose a special set $M_1$ of $k$ edges from $M\cup E(x)$, and then let $Q_{AB}$ be the graph induced by these edges as well as all edges of $Q[A], Q[B]$. Note that some of our $k$ edges may be chosen from $Q[A] \cup Q[B]$. 

First suppose we are in case 1. Then $|M|=\lfloor \frac{n}{2} - 0.1 \ve n \rfloor>k$. If $n$ is even, we let $M_1$ be a set of any $k$ edges from $M$. If $n$ is odd, then $x\in V(Q)$ with $|E(x)|= n-2|M|$; we choose for $M_1$ all edges in $E(x)\cap E(Q[B])$ and then $k-|E(x)\cap E(Q[B])|$ edges from $M$.  

Now suppose that we are in case 2. We know that $|M|\in\{\lfloor \frac{n}{2} - 0.1 \ve n \rfloor, n-1-\Delta(G) \}$. In the former case, we have $|M|>k$ and we choose the edges of $M_1$ as in case 1. Assume now that $|M|=n-1-\Delta$. Then
$$|E(x)|+|M|=n-2|M|+|M|=n-|M|=\Delta+1>k.$$

Suppose first that $n$ is odd. We know that  $x\in V(Q)$ and by property (P3) of the bipartition,
\begin{eqnarray*}|E(x)\cap E(Q[B]))| &\leq& \tfrac{1}{2}\left(|E(x)|+ n^{2/3}\right)=\tfrac{1}{2}(n-2|M|+n^{2/3})\\
&= & \tfrac{1}{2}(2\Delta+2-n+n^{2/3}) < \tfrac{1}{2}(2\Delta+2-(\tfrac{4\Delta}{3})+n^{2/3})<k.
\end{eqnarray*}
There are at most $\xi n$ edges of $M$ whose endpoints are not partnered and thus may lie in the same side of the bipartition $(A, B)$.  So in fact, 
$$     |M\cap (E(Q[A]\cup Q[B]))|+|E(x)\cap E(Q[B])|  <k. $$
For the set $M_1$, we can therefore take all the edges of $M\cap (E(Q[A]\cup Q[B]))$ and all the edges of $E(x)\cap E(Q[B])$; we take the rest of the $k$ edges from $M\cap E(Q[A, B])$ and then from $E(x)\cap E(Q[A, B])$, as necessary

Suppose now that $n$ is even. If $k\leq |M|$, then we choose for $M_1$ all edges in $M\cap (E(Q[A])\cup E(Q[B]))$ (at most $\xi n$ of them) and then take the rest of the $k$ edges from $M\cap E(Q[A, B])$. If $k>|M|$ we obtain $M_1$ by taking all edges of $M$ and $k-|M|$ edges from $E(x)$.

Note that we have defined $Q_{AB}$ so that it is a subgraph of $Q$, except in one exceptional situation: when $n$ is even (so $x\not\in V(Q)$), and $k>|M|$. In this special situation we let $S_0=\{x\}$, and otherwise we let $S_0=\emptyset$. In general we have $V(Q_{AB})=V(Q)\cup S_0$.\\

\noindent \underline{{\sc Coloring Phase.}}\\

\noindent \textbf{Step 1: Edge color $Q_{AB}$.}

By property (P3) and by how we constructed $Q$, we know that for any  $w\in V(Q)\setminus \{x\}$ we have
\begin{eqnarray}\label{eqn:partition-neighbor-difference}
    \left||N^A_Q(w)|-|N^B_Q(w)|\right| \le  (\tfrac{m}{2})^{2/3} +1; 
\end{eqnarray}
when $x\in V(Q)$ (i.e. when $n$ is odd) this difference may increase by more than one when we construct $Q$ from $G_1$. The following claim tells us about the degrees of vertices in $Q_{AB}$.\\

\begin{claim}\label{claim:G_AB-maximum-min-degree} The following statements hold. 
\begin{enumerate}[(1)]
    \item  For the vertex $v_1$, we have \begin{numcases}{}
\Delta/2 -1.1\xi n      \le d_{Q_{AB}}(v_1) \le \Delta/2 +1.1\xi n & in Case 1; \nonumber  \\
\Delta/2 -n^{2/3}      \le d_{Q_{AB}}(v_1) \le \Delta/2 +n^{2/3} & in Case 2 when $n$ odd.  \nonumber 
 \end{numcases} 
 Furthermore, we have $d_Q(v_1) \le \Delta+1$. 
 \item For any vertex $v\in V(Q_{AB}) \setminus U$, we have 
 $$
\Delta/2 -1.1 \xi n  \le d_{Q_{AB}}(v) \le \Delta/2 +n^{2/3}.  
 $$
 \item  For any vertex $v\in U$, where $v\ne v_1$ in Case 1 or Case 2 with $n$  odd, we have 
  $$
\delta/2 -n^{2/3}  \le d_{Q_{AB}}(v) \le \Delta/2 +n^{2/3}.  
 $$ 
\end{enumerate}
\end{claim}

\begin{proofc} First we prove (1). 
 
Note that 
\begin{eqnarray*}
    d_{Q_{AB}} (v_1) &=&e_{Q_{AB}}(v_1,A) +e_{Q_{AB}}(v_1,B),  \\
    e_{Q_{AB}}(v_1,B) &=&e_{G_1[B]}(v_1,B)+q_1,  \quad \text{and} \\
   \frac{1}{2}(d_{G_1}(v_1)-\tfrac{m}{2})^{2/3}) &\le& e_{G_1[B]}(v_1,B) \le \frac{1}{2}(d_{G_1}(v_1)+\tfrac{m}{2})^{2/3}). 
\end{eqnarray*}

We first prove that $e_{Q_{AB}}(v_1,A) \le 1$. 
If  $|M| =\lfloor \frac{n}{2} -0.1\ve n \rfloor$, then as $|M|>k$, by our construction of $Q_{AB}$, we have 
$e_{Q_{AB}}(v_1,A) \le 1$.  
If $|M| = n-1-\Delta$, then $G$ is in Case 2 and $n$ is odd. By our construction of $M_1$, we have 
$$
e_{Q_{AB}}(v_1,A)  =  \max\{k -|E(x) \cap E(Q[B])| - |M|, 0 \}. 
$$
 Since  $|E(x) \cap E(Q[B])| =|N^B_{G_1}(x)| \ge \frac{1}{2}(d_{G_1}(x)-(\tfrac{m}{2})^{2/3}) = \frac{1}{2} (2\Delta+2- n -(\tfrac{m}{2})^{2/3})$, and $|M| =n-1-\Delta$, it follows that  $|M|+|E(x) \cap E(Q[B])| 
 \ge  \frac{n}{2}- \frac{1}{2}(\tfrac{m}{2})^{2/3}>k$. 
 Thus $e_{Q_{AB}}(v_1,A)  =0$. 

Since $e_{Q_{AB}}(v_1,A) \le 1$, we have 
$$
 \frac{1}{2}(d_{G_1}(v_1)-(\tfrac{m}{2})^{2/3})+q_1 \le d_{Q_{AB}} (v_1) \le \frac{1}{2}(d_{G_1}(v_1)+(\tfrac{m}{2})^{2/3}) +q_1 +1.
 $$
When $q_1 \le 1$ or $q_1 = \lceil \frac{1}{2}(\Delta - d_{G_1}(v_1)) \rceil$, 
Claim~\ref{claim:G_AB-maximum-min-degree}(1) follows 
easily. So we assume that  $q_1=e(G_1[A])-e(G_1[B])$. 
This implies that $G$ is in Case 1. By Claim~\ref{claim:Case1-n-odd-e(A)>e(B)} on the lower and upper bound on $q_1$, we again get the desired bounds. 

Next we show that $d_Q(v_1) \le \Delta+1$. We have that $d_Q(v_1) =d_{G_1}(v_1)+q_1$ and $d_{G}(v_1) =\delta$. When $q_1=0$,  the assertion holds trivially. 
Thus we assume that $q_1\ge 1$. If $q_1=1$, then by our definition of $q_1$, we know that $G$ is not regular. 
Thus $d_Q(v_1)  \le d_{G}(v_1)+1+q_1 \le \Delta+1$. 
Therefore we assume that $q_1=e(G_1[A]) -e(G_1[B])$ 
or $q_1= \lceil \frac{1}{2}(\Delta -d_{G_1}(v_1)) \rceil$.

If $q_1=e(G_1[A]) -e(G_1[B])$, by 
 Claim~\ref{claim:Case1-n-odd-e(A)>e(B)} on the lower and upper bound on $q_1$, we get 
 $d_Q(v_1)  \le d_{G_1}(v_1) + \frac{1}{2}(\Delta+1+2\xi n -d_{G_1}(v_1))  =\frac{1}{2}(\Delta+1+2\xi n+d_{G_1}(v_1))$.   Under the case that $q_1=e(G_1[A]) -e(G_1[B])$,
 we have $d_{G_1}(v_1) < \Delta+1-2\xi n$ or $d_{G_1}(v_1) =n-2|M| \le 0.2 \ve n+1$. In both cases, we get $d_Q(v_1)  \le  \Delta+1$. 
 
If $q_1= \lceil \frac{1}{2}(\Delta -d_{G_1}(v_1)) \rceil$, then 
   $G$ is in Case 2 and $n$ is odd.   Thus we have $d_{G_1}(v_1) =2\Delta+2-n$, and so 
 we get $d_Q(v_1) =d_{G_1}(v_1)+\lceil \frac{1}{2}(\Delta -d_{G_1}(v_1)) \rceil   \le  \frac{1}{2}(\Delta+d_{G_1}(v_1))+1  =\frac{1}{2}(3\Delta+2-n)+1 \le \Delta+1$ as
 $n\ge \Delta+2$. 
  
For (2) with $v\not\in U$, we  have $d_{G_1}(v) \ge \Delta-\xi n$. Thus, 
  by~\eqref{eqn:partition-neighbor-difference} and the construction of $Q$, we have 
$$
 \frac{1}{2}\left (d_{G_1}(v)- ((m/2)^{2/3}+1) \right) \le d_{Q_{AB}}(v) \le \frac{1}{2}\left(d_{G_1}(v)+ ((m/2)^{2/3}+1)\right),  
$$
and so 
$$
\frac{1}{2} \Delta -1.1 \xi n  \le d_{Q_{AB}}(v) \le \frac{1}{2} \Delta +n^{2/3}. 
$$

For (3) with $v\ne v_1$ in Case 1 and $v\ne v_1$ in Case 2 when $n$ is odd, we have $d_{G_1}(v) \ge \delta$, 
and so  the desired lower and upper bounds on $d_{Q_{AB}}(v)$ are calculated the same way as in (2).  
\end{proofc}

We can now get the coloring of $Q_{AB}$ that we are looking for.

\begin{claim} \label{claim:existence-k-coloring}
    The multigraph $Q_{AB}$ has a $k$-edge-coloring such that all the edges of $M_1$ are colored differently. 
\end{claim}

\begin{proofc} Note that $k= |M_1| \ge \Delta(Q_{AB})+4$. 
 We let $Q^s_{AB}$ be the underlying simple graph of 
$Q_{AB}$.
In Case 1 when $d_{G_1}(v_1)< \Delta-2\xi n$ or in Case 2 with $n$ odd,  we let $J=E_{Q_{AB}}(v_1, B\setminus \{v_1\}) \cup M_1$ and $J_0 =M_1$.  
In Case 1 when $d_{G_1}(v_1) \ge  \Delta-2\xi n$, we let $J=(E(Q_{AB}) \setminus E(Q^s_{AB}) )\cup M_1$ and let $J_0 = M_1$. Note that these are the edges added between pairs of vertices in $B_0$ in Step 2 along with $M_1$. In Case 2 with even $n$, we let $J=J_0 =M_1$.  Note that possibly other than $v_1$, all  other vertices of $J$ have degree at most 2 in $Q_{AB}[J]$. 
In order to apply Lemma~\ref{lemma:matching-extension}(a), it suffices to show that 
$Q_{AB}[J]$ has a $k$-edge-coloring such that each edge of $J_0$ is colored differently.

In Case 1 when $d_{G_1}(v_1)< \Delta-2\xi n$ or in Case 2 with odd $n$, 
we first pre-color all edges  of $M_1$ with all different colors. 
Now each edge from $E_{Q_{AB}}(v_1, B\setminus \{v_1\})$ is adjacent in $Q_{AB}[J]$ 
with at most $|M_1\cap E(x)| +1$ edges. 
As $|E_{Q_{AB}}(v_1, B\setminus \{v_1\}) \setminus E(x)| \le |E(v_1)|\le \Delta(Q_{AB}) -|M_1\cap E(x)|$, we can extend the pre-coloring on edges of $M_1$ to all edges of $J$. 
In Case 1 when $d_{G_1}(v_1) \ge  \Delta-2\xi n$ or in Case 2 with even $n$, as $ \Delta(Q_{AB}[J]) \le 2$, we can first pre-edge color edges of $J_0$ such that each of them is colored with a   different color. Then we can extend the edge coloring to all edges of $J$. 
\end{proofc}

Let $\varphi_0$ be the $k$-edge-coloring of $Q_{AB}$ we get from Claim~\ref{claim:existence-k-coloring}. We can make some additional assumptions about $\varphi_0$, as follows.

Suppose first that we are in case 1. By Lemma~\ref{lem:equitable-coloring-precolored-edges}, we may choose our $k$-edge-coloring $\varphi_0$  such that
\begin{equation}\label{leq4}
\left||\pbar_0^{-1}(i)| -|\pbar_0^{-1}(j)| \right| \le 5
\end{equation}
for any  $i,j\in [1,k]$.  For any vertex $v\in V(Q_{AB})$,  $|\pbar_0(v)| =k-d_{Q_{AB}}(v)$. 
If $v\not\in U$, 
by Claim~\ref{claim:G_AB-maximum-min-degree}, we get that
\begin{equation}\label{pbar_0v}
|\pbar_0(v)|\leq k-\left(\frac{1}{2}\Delta -1.1 \xi n\right) \le  2.3\xi n. 
\end{equation}
 When $v \in U$, by  Claim~\ref{claim:G_AB-maximum-min-degree}, we have 
\begin{eqnarray}
  |\pbar_0(v)|\leq k-\left(\frac{1}{2}\delta-n^{2/3}\right) <\frac{1}{4} n.  \label{eqn:miss-color-W-vertex}
\end{eqnarray}
Overall, we have 
\begin{equation}\label{eqn:total-number-missing-colors-Cases1-2}
    4n\le \sum_{v\in V(Q_{AB})} |\pbar_0(v)| \le (m-|U|)(2.3\xi n)+|U|(n/4) <3 \xi n^2. 
\end{equation}
Since $Q_{AB}$ has exactly $m$ vertices, this means that the average number of vertices not seeing any particular color is at most
$$
\frac{3 \xi n^2}{k}  < 12\xi n-5. 
$$
By~(\ref{leq4}), we get  
\begin{equation}\label{eqn1}
	|\pbar_0^{-1}(i)|< 12\xi n \quad \text{for each $i\in [1,k]$}.
\end{equation}

\medskip 

Now suppose that we are in case 2. We  modify $\varphi_0$ as follows. Let 
$$U^*=\{u\in V(Q): \Delta -d_{Q}(u) \ge  3.5n^{2/3}\}.$$
If there exist distinct $u,v\in U^*\cap A$  or distinct $u,v \in U^*\cap B$ such that $\pbar_0(u)\cap \pbar_0(v)\ne \emptyset$, we add an edge joining $u$ and $v$ and color the new edge by a color in $\pbar_0(u)\cap \pbar_0(v)$.

   For convenience, when the edge coloring $\varphi_0$ is updated, we still call it $\varphi_0$. 
	We iterate this process of adding and coloring edges and call the  multigraphs  resulting from $Q[A]$ and $Q[B]$, respectively, $Q^*_A$ and $Q^*_B$, and call 
	$Q^{**}$ the union of $Q^*_A$, $Q^*_B$ and $Q[A,B]$, and $Q^*_{AB}$ the resulting multigraph of $Q_{AB}$.  

    For the current edge coloring $\varphi_0$ of $Q^*_{AB}$,  
the following statement holds:  $\pbar_0(u)\cap \pbar_0(v)=\emptyset$
for any two distinct $u,v\in U^*\cap A$ or any two distinct $u,v\in U^*\cap B$.  
Therefore, 
\begin{eqnarray}
    \sum_{u\in U^*\cap A}|\pbar_0(u)|&=&\sum_{u\in U^*\cap A} (k-d_{Q^*_A}(u)) \le k,  \nonumber \\
    && \label{eqn:total-number-missing-colors-Case3}\\
    \sum_{u\in U^*\cap B}|\pbar_0(u)|&=&\sum_{u\in U^*\cap B} (k-d_{Q^*_B}(u)) \le k.   \nonumber
\end{eqnarray}

\begin{claim}\label{claim:vertex-degree-in-Q^*}
We have
\begin{numcases}{d_{Q^{**}}(u) \le}
    \Delta-0.2n^{2/3} & if $u\in U^*$; \label{eqn:degree-of-U-vertex-in Q^**} \\ 
    \Delta-0.4\xi n & if $u\in U$.  \nonumber
\end{numcases}
\end{claim}

\proof 
For any $u\in U^*\cap A$, we have 
$$
d_{Q^{**}}(u) \le 
k +e_Q(u, B)  \le  \frac{1}{2}\Delta+n^{2/3} +5+\frac{1}{2}(\Delta-3.5n^{2/3}+n^{2/3}) \le \Delta-0.2n^{2/3}.
$$
Similarly, we have $d_{Q^{**}}(u) \le \Delta-0.2n^{2/3} $ for any $u\in U^*\cap B$.  

When $u\in U\cap A$, then we have 
$$
d_{Q^{**}}(u) \le 
k +e_Q(u, B)  \le  \frac{1}{2}\Delta+n^{2/3} +5+\frac{1}{2}(\Delta-\xi n+n^{2/3}) \le \Delta-0.4\xi n.
$$
Then case when $u\in U\cap B$ is proved similarly. 
 \qed

Again by  Lemma~\ref{lem:equitable-coloring-precolored-edges}, we may choose our $k$-edge-coloring $\varphi_0$  under the modification above 
that satisfies~\eqref{leq4}.  Under this assumption on  $\varphi_0$, 
 it is possible that $\pbar_0(u)\cap \pbar_0(v) \ne \emptyset$
for some distinct  $ u,v\in U^*\cap A$ or distinct $u,v\in U^*\cap B $. 
However  the inequalities in~\eqref{eqn:total-number-missing-colors-Case3} still hold.  
We now examine the number of missing colors at each vertex of $Q^*_{AB}$. 

For any vertex $v\in V(Q_{AB})\setminus  U^*$, we have $d_Q(u) > \Delta -3.5n^{2/3}$. 
Thus $d_{Q_{AB}}(u) \ge \frac{1}{2}(d_Q(u)-n^{2/3}) \ge \frac{1}{2} \Delta-2.3 n^{2/3}$, and so 
\begin{eqnarray}
    |\pbar_0(u)| &=& k-d_{Q_{AB}}(u)  \nonumber \\
    & \le & \frac{1}{2}\Delta+n^{2/3}+5 -(\frac{1}{2} \Delta-2.3 n^{2/3}) <3.4n^{2/3}. \label{eqn:number-missing-color-Case2}
\end{eqnarray}
Therefore, 
\begin{equation}\label{eqn:total-number-missing-colors-Case3-1}
   4n \le  \sum_{v\in V(Q_{AB}-S_0)} |\pbar_0(v)| \le (n-|U^*|)(3.4 n^{2/3})+2k<3.5 n^{5/3}. 
\end{equation}
Since $Q_{AB}-S_0$ (recall that $S_0$ is defined in Step 3) has exactly $m \le n+1$ vertices, this means that the average number of vertices of $Q_{AB}-S_0$ not seeing any particular color is at most
$$
\frac{3.5 n^{5/3} }{k} < 14 n^{2/3}-5. 
$$
By~(\ref{leq4}), we get  
\begin{equation}\label{eqn:case3-average-number-of-vertices-missing-i}
	|\pbar_0^{-1}(i)|< 14 n^{2/3} \quad \text{for each $i\in [1,k]$}.
\end{equation}

For notational uniformity, we will still use $Q$ for $Q^{**}$, and use $Q[A,B]$ for $Q^*[A,B]$ (we will not distinguish between them from now onwards).

We denote by $\varphi$ the current partial edge coloring of $Q$ (whether we are in case 1 or case 2). In this way, we can refer back the number of colors missing at vertices in Step 1 by using the notation $\varphi_0$.  \\

\noindent \textbf{Step 2: Extending the $k$ color classes.}

We start this step by defining
 \begin{numcases}{W=}
     U & in Case 1;  \nonumber \\ 
     \{w\in U: |\pbar_0(w)|  \ge  3.4n^{2/3} \} & in Case 2. \nonumber 
 \end{numcases}
 In Case 1, we have $|W| =|U|<\xi n$. 
 In Case 2, by~\eqref{eqn:number-missing-color-Case2}, we know 
 that $W\subseteq U^*$. Then by~\eqref{eqn:total-number-missing-colors-Case3},
 we get $|W| \le \frac{2k}{3.4n^{2/3}} < n^{1/3}$. 
The vertices of $W$ will receive particular attention in this step. 

During this step, by swapping colors along alternating paths, we will increase the size of the $k$ color classes obtained  in the 
previous step until each color class is a perfect matching  of $Q$ in Case 1, and saturates $V(Q)\setminus (U\setminus W)$ in Case 2. During this process, we will maintain
the colors on the edges of $M_1$. We will however, uncolor some of the edges of $Q_A$ and $Q_B$. 

Let $R_A, R_B$ be the the subgraphs induced by all uncolored edges in $Q_A, Q_B$, respectively. These multigraphs will both initially be empty. 
As we progress through this step however,  
one or two edges will be added to each of $R_A$ and $R_B$ 
each time we swap colors on an alternating path.  
As the structure of $Q$ in Case 1 is quite different with that in Case 2, 
and our approach in dealing with these two   cases are   different too, 
we define the following parameters according to different cases:
\begin{numcases}{s=}
3  \xi n^2    &  \text{in Case 1}; \nonumber  \\
&\\ 
3.5 n^{5/3}    &  \text{in Case 2}.  \nonumber 
    \end{numcases}

    \begin{numcases}{r=}
 \lceil \xi^{1/2} n  \rceil   &  \text{in Case 1};  \nonumber  \\
  &\\ 
 \lceil  n^{5/6}  \rceil   &  \text{in Case 2}.  \nonumber 
    \end{numcases}
We will tightly control these uncolored edges, and will ensure that the following conditions are satisfied after the completion of this step:
\begin{enumerate}
	\item[(C1)] $|E(R_A)|, |E(R_B)|  \le s$. Additionally,  $|E(R_A)|= |E(R_B)|$ for Case 1 ;

		\item[(C2)] $\Delta(R_A),\Delta(R_B)  < r$;

	\item[(C3)] Each vertex  $v$ of $V(Q)\setminus W$ is incident with at most  $|\pbar_0(v)|+r$ colored edges of $Q[A, B]-E(Q_{AB})$; 
    and each vertex $v$ of $W$ is incident with at most  $|\pbar_0(v)|$ colored edges of $Q[A, B]-E(Q_{AB})$; 
\end{enumerate}

At any point during our process, we say that an edge  $e=uv$ is \emph{good} if $e\not \in E(R_A) \cup E(R_B)$ and the degree of $u$
and $v$ in both $R_A$ and  $R_B$ is less than $r-1$.  In particular, a good edge can be added to $R_A$ or $R_B$ without violating condition (C2).

By Lemma~\ref{lem:parity}, we know that $|\pbar^{-1}(i)| \equiv |V(Q_{AB})|\pmod{2}$   for every color $i\in [1,k]$. When $V(Q_{AB})=V(Q)$ we get that $|V(Q_{AB})|=m$ even, so since $(A, B)$ is a partition of $V(Q)$, the quantity $|\pbar^{-1}_A(i)|-|\pbar^{-1}_B(i)|$ is also even. We can pair up as many vertices as possible from $\pbar^{-1}_A(i),\pbar^{-1}_B(i)$, and then pair up 
the remaining unpaired vertices from  $\pbar^{-1}_A(i)$ or $\pbar^{-1}_B(i)$.  
When $V(Q_{AB})\neq V(Q)$ we get that $|V(Q_{AB})|$ is odd, and we can pair up as many vertices as possible from $\pbar^{-1}_A(i),\pbar^{-1}_B(i)$, and then pair up 
the remaining unpaired vertices from  $\pbar^{-1}_A(i)$ or $\pbar^{-1}_B(i)$ 
but with exactly one vertex leftover on one side. 
We call each of these pairs of vertices a \emph{missing-common-color pair} or \emph{MCC-pair} with respect to the color $i$. In fact, given any MCC-pair $(a,b)$ with respect to some color $i$, we will show that the vertices $a, b$ are joined by some path $P$ (with at most 7 edges), where both ends of the path have an uncolored edge from $Q[A, B]$, and then the path alternates between good edges (colored $i$) and uncolored edges. Given such a path $P$, we will swap along $P$ by giving all the uncolored edges $i$, and uncoloring all the colored edges. After this, both $a$ and $b$ will be incident with edges colored $i$, and at most three good edges will be added to $R_A\cup R_B$. Given the alternating nature of the path $P$, at most two good edges will be added to any one of $R_A, R_B$.
The main difficulty in this Step 2 is to show that such path $P$ exists. Before we do this, let us show that conditions (C1)-(C3) can be guaranteed at the end of this swapping process, that is, at the end of this step. In fact, 
we will only ever add good edges to and $R_A$ and $R_B$, so condition (C2) will  hold automatically.  It remains then only to check (C1) and (C3).

Let us first consider condition (C1).  By~\eqref{eqn:total-number-missing-colors-Cases1-2} and~\eqref{eqn:total-number-missing-colors-Case3-1}, 
the sum of
missing colors  under $\varphi_0$ taken over all vertices of $Q$ is at most  $s$. 
Thus there are at most   $s/2$ MCC-pairs. As we add at most two edges to each of $R_A$ and $R_B$ for each MCC-pair,   we added at most $s$ edges to each of $R_A$
and $R_B$.  Thus, condition (C1) is satisfied. 
As we said above, when we are in case 1, we will extend each of the $k$ color classes of $\varphi_0$ to a perfect matching of $Q$ in this step. Thus, in case 1,
each of $Q[A]$ and $Q[B]$ will contain the same number of edges from each color class by the end of this step. Since $e(Q[A])=e(Q[B])$ in case 1, by our construction of $Q$, we will have $|E(R_A)|= |E(R_B)|$ at this end of this step.

We  now show that Condition (C3)  will also be satisfied at the end of Step 3. In the process of Step 3,   the number of newly colored edges of $Q[A, B]$ that are incident with a vertex   $u\in V(Q)$, where $u \not\in W$,   will equal  the number of our chosen alternating paths containing $u$. The number of such alternating paths of which $u$ is the first vertex will equal the number of colors missing from $u$ at the end of Step 2. When $u\in W$, we will make sure that $u$ is not used in any alternating path where it is not an endpoint, so the number of newly colored edges of $Q[A, B]$ that are incident with $u$
will be $|\pbar_0(u)|$.  If $u\not\in W$,  the number of alternating paths in which $u$ is not the first  vertex will  equal  the degree of $u$ in $R_A\cup R_B$, and so will be less than $r$. Thus Condition (C3)   will be satisfied. 

Let $i\in[1,k]$. We will describe how to find the alternating paths for the MCC-pairs with respect to $i$. There is one $i$-colored edge in $M_1$, say $e_i$ with endpoints $V(e_i)$, and we will be sure to avoid it in all of the alternating paths joining our MCC-pairs. For each vertex in an MCC-pair, we will see how to construct a short (two-edge) alternating path from it that avoids $e_i$. We will see that there are so many choices for such paths that we can connect them to link each MCC-pair. 

Let $v\in V(Q)$. Suppose that $v\in S_v\in\{A, B\}$, and let $T_v$ be the other one of $A, B$. Define $N_1(v)$ to be the set of all vertices in $T_v\setminus (V(e_i)\cup  W)$ that are joined to $v$ by an uncolored edge, and are incident to a good edge colored $i$.  Let $N_2(v)$ be the set of vertices in $T_v\setminus (V(e_i)\cup W)$ that are joined to a vertex of $N_1(v)$ by a good edge of color $i$. Note that we have $|N_1(v)|=|N_2(v)|$, but some vertices may be in $N_1(v)\cap N_2(v)$. 

There are fewer than $s$ edges in $R_B$ (by (C1)), so there are fewer than $2s/(r-1)$ vertices of degree at least $r-1$ in $R_B$. 
Each non-good edge is incident with one or two 
vertices of $R_B$ through the color $i$. Thus  the number of vertices in $B$ incident with a non-good edge colored by $i$
is less than  $2\times(2s/(r-1))=4s/(r-1)$. 
In addition, by~\eqref{eqn1} and~\eqref{eqn:case3-average-number-of-vertices-missing-i} there are fewer than $4s/n$
vertices in $B$ that are missed by the color $i$. So the number of vertices in $B$ which are either incident to a non-good edge colored $i$, missing the color $i$, in $V(e_i)\cup W$, or incident to a good edge colored $i$ whose other endpoint is in $W$ is less than
\begin{equation}\label{eqn3}
	4s/(r-1)+4s/n+2|W|+2<5s/r. 
\end{equation} 
Let  \begin{numcases}{t=}
 1.1 \xi  n   &  \text{in Case 1};  \nonumber  \\
  &\\ 
   n^{2/3}  &  \text{in Case 2}.  \nonumber 
    \end{numcases}
By symmetry, the number of vertices that
are either incident to a non-good edge colored $i$, missing the color $i$, in $V(e_i)\cup W$, or incident to a good edge colored $i$ whose other endpoint is in $W$ is less than
$5s/r$. Therefore for any $v\in V(Q) \setminus W$, we have
\begin{eqnarray}
|N_1(v)|=|N_2(v)|  & \ge& d_{Q[A, B]}(v) -(|\pbar_0(v)|+r) -5s/r \nonumber  \\
 &\ge & \tfrac{1}{4} (1+\ve)n-n^{2/3}- (|\pbar_0(v)|+r) -5s/r \nonumber \\
 &>&\tfrac{1}{4}(1+0.5\ve) n, \label{eqn:N-M-sizea}.
\end{eqnarray}
where  in the inequalities to get~\eqref{eqn:N-M-sizea},    the number $|\pbar_0(v)|+r$ is the maximum number of 
colored edges of $Q[A, B]$ that are incident with $v$ at the end of Step 3 (as per (C3)). When $v\in W$ with $v\ne v_1$ when $n$ is odd, we have  $|\pbar_0(v)| \le \frac{1}{4} n$ by~\eqref{eqn:miss-color-W-vertex}. 
Thus,
\begin{eqnarray} 
|N_1(v)|=|N_2(v)| &\ge  &\tfrac{1}{4} (1+\ve)n-n^{2/3}- 5s/r-\tfrac{1}{4}n  \ge  0.5\ve n.  \label{eqn:size-of-N_1(v_0)} 
\end{eqnarray} For the vertex $v_1$ when $n$ is odd, by our choice of the matching $M$ and Claim~\ref{claim:G_AB-maximum-min-degree}(1) that $|\pbar_0(v)|\le 2.1 t$, we have 
\begin{eqnarray} 
|N_1(v_1)|=|N_2(v_1)| &\ge  &0.2\ve n-n^{2/3}-5s/r-2.1t  \ge  0.1\ve n.  \label{eqn:size-of-N_1(v_0)0} 
\end{eqnarray}

Before we show how to build an alternating path between each MCC-pair, we will first do some preliminary alternating path switches so as to refine our selection of the MCC-pairs. 

If $n$ is odd, we don't want $v_1=x$ to be in any MCC-pair, so we will first change this.  Recall that $v_1 \in B$ by our assumption 
and $|\pbar_0(v_1)|\le 2.1 t$ by Claim~\ref{claim:G_AB-maximum-min-degree}(1). 
For each $i\in \pbar_0(v_1)$, we choose $w_1\in N_1(v_1)$ and $w_2\in N_2(v_2)$. 
Then we color $v_1w_1$ by $i$ and uncolor $w_1w_2$.  Since all the $k$ edges in $M_1$ are colored by $k$ distinct colors, 
$i\in \pbar_0(v_1)$ implies that $i$ is used on an edge $e$ of $M\cap M_1$.  We uncolor $e$, and update $M_1$ by deleting $e$
and adding $v_1w_1$ (the uncolored edge $e$ is placed back as one of the uncolored edges of $M$). Note that the new edge set $M_1$ still has exactly $k$ edges and each of them is colored by a different color. 
In this process, we make sure that the edge $w_1w_2$ that was  uncolored  (and hence added to 
$R_A$) is a good edge. This is possible by~\eqref{eqn:size-of-N_1(v_0)0}.  As $|\pbar_0(v_1)|\le 2.1 t$,  $R_A\cup R_B$
contains at most $2.1t$ edges of $M$. 

Suppose now that we are in  Case 2. Then for each color $i$, if $i$ is missing at an even number of 
vertices from $V(Q)\setminus (U\setminus W)$, we only pair up those vertices as MCC-pairs. If $i$ is missing at an odd  number of 
vertices from $V(Q)\setminus (U\setminus W)$ and is also missing at a vertex of $u\in (U\setminus W)$, then  
we pair up vertices of $(V(Q)\setminus (U\setminus W)) \cup \{u\}$ that are all missing $i$ as MCC-pairs. If $i$ is missing at an odd  number of 
vertices from $V(Q)\setminus (U\setminus W)$ but is not missing at any vertex of $(U\setminus W)$, then we choose $u\in (U\setminus W)$ and $j\in \pbar_0(u)$ and do a Kempe-change at a maximal $(i,j)$-alternating path $P$ starting at $u$. After this, we can now either pair up vertices of $V(Q)\setminus (U\setminus W)$
that are all missing $i$ (this happens if the other endvertex of $P$ is contained in $V(Q)\setminus (U\setminus W)$) or pair up vertices of $(V(Q)\setminus (U\setminus W)) \cup \{u\}$ that are all missing $i$ as MCC-pairs with respect to $i$.  

Suppose now that we have an MCC-pair $(a,b)$ with $\{a,b\}\cap W = \emptyset$. We will replace the pair $(a,b)$ by $(a^*,b^*)$ such that $\{a^*,b^*\}\cap W=\emptyset$ (in order to help streamline the discussion below).  Precisely, we will implement the following 
 operations to vertices in $W$ when $n$ is even and to vertices of $W\setminus \{v_1\} $ when $n$ is odd. 
 For any vertex $a\in W\cap A$, and for each color $i\in \pbar_0(a)$, we take an edge $b_1b_2$ with $b_1\in N_1(a)$ and $b_2\in N_2(a)$ such that $b_1b_2$ is colored by $i$, where the edge $b_1b_2$ exists by~\eqref{eqn:N-M-sizea} and~\eqref{eqn:size-of-N_1(v_0)} and the fact that $|N_2(a)|=|N_1(a)|$.   Then we exchange the path $ab_1b_2$ by coloring $ab_1$
 with $i$ and uncoloring the edge $b_1b_2$ (See Figure~\ref{f1}(a)).
  After this, the edge $ab_1$  of $Q[A,B]$ is now colored by $i$, and the uncolored edge $b_1b_2$ is added to $R_B$.
 We then update the original MCC-pair that contains $a$ with respect to the color $i$ by replacing the vertex $a$ with $b_2$. 
 We do this at the vertex $a$ for every color $i\in \pbar_0(a)$ and then repeat the same process for every vertex in $W\cap A$. 
 Similarly, 
 for any vertex $b\in W\cap B$ when $n$ is even and $b\in W\cap B$ with $b\ne v_1$ when $n$ is odd, and for each color $i\in \pbar_0(b)$, we take an edge $a_1a_2$ with $a_1\in N_1(b)$ and $a_2\in N_2(b)$ such that $a_1a_2$ is colored by $i$, where the edge $a_1a_2$ exists by~\eqref{eqn:N-M-sizea} and~\eqref{eqn:size-of-N_1(v_0)}  and the fact that $|N_1(b)|=|N_2(b)|$.   Then we exchange the path $ba_1a_2$ by coloring $ba_1$
 with $i$ and uncoloring the edge $a_1a_2$. The same, we update the original MCC-pair that contains $b$ with respect to the color $i$ by replacing the vertex $b$ with $a_2$. 

We are now finally able to describe, for each MCC-pair $(u, v)$, how to build 2-edge-paths from each of $u, v$ so that these paths connect as desired. This path construction will depend on whether $u,v$ are on the opposites sides of the bipartition $(A, B)$ or not; our paths will look like those pictured in Figure 1(b) and 1(c), respectively.

\begin{figure}[!htb]
	\begin{center}
		
		\begin{tikzpicture}[scale=0.6]
		
			\begin{scope}[shift={(-10,0)}]
		\draw[rounded corners, fill=white!90!gray] (4, 0) rectangle (8, 2) {};
		
		\draw[rounded corners, fill=white!90!gray] (4, -4) rectangle (8, -2) {};
		
		{\tikzstyle{every node}=[draw ,circle,fill=white, minimum size=0.5cm,
			inner sep=0pt]
			\draw[black,thick](5,1) node (c)  {$a$};
			\draw[black,thick](5,-3) node (d)  {$b_{1}$};
			\draw[black,thick](7,-3) node (d1)  {$b_{2}$};
			
		}
		\path[draw,thick,black!60!white,dashed]
		(c) edge node[name=la,pos=0.7, above] {\color{blue} } (d)
		;
		
		\path[draw,thick,black]
		(d) edge node[name=la,pos=0.7, above] {\color{blue} } (d1)
		;
		\node at (3.2,1) {$A$};
		\node at (3.2,-3) {$B$};
		\node at (6,-4.8) {$(a)$};	
		\end{scope}	
		
		\begin{scope}[shift={(5,0)}]
		\draw[rounded corners, fill=white!90!gray] (-6, 0) rectangle (0, 2) {};
		
		\draw[rounded corners, fill=white!90!gray] (-6, -4) rectangle (0, -2) {};
		
		{\tikzstyle{every node}=[draw ,circle,fill=white, minimum size=0.5cm,
			inner sep=0pt]
			\draw[black,thick](-5,1) node (a)  {$a$};
			\draw[black,thick](-3,1) node (a1)  {$a_1$};
			\draw[black,thick](-1,1) node (a2)  {$a_2$};
			\draw[black,thick](-5,-3) node (b)  {$b$};
			\draw[black,thick](-3,-3) node (b1)  {$b_1$};
			\draw[black,thick](-1,-3) node (b2)  {$b_2$};
		}

		\path[draw,thick,black!60!white,dashed]
		(a) edge node[name=la,pos=0.7, above] {\color{blue} } (b1)
		(a2) edge node[name=la,pos=0.7, above] {\color{blue} } (b2)
		(b) edge node[name=la,pos=0.6,above] {\color{blue}  } (a1)
		;
		
		\path[draw,thick,black]
		(a1) edge node[name=la,pos=0.7, above] {\color{blue} } (a2)
		(b1) edge node[name=la,pos=0.7, above] {\color{blue} } (b2)
		;
		
		\node at (-3,-4.8) {$(b)$};


		\draw[rounded corners, fill=white!90!gray] (1, 0) rectangle (9, 2) {};
		
		\draw[rounded corners, fill=white!90!gray] (1, -4) rectangle (9, -2) {};
		
		{\tikzstyle{every node}=[draw ,circle,fill=white, minimum size=0.5cm,
			inner sep=0pt]
			\draw[black,thick](2,1) node (c)  {$a$};
			\draw[black,thick](4,1) node (c1)  {$a_{2}$};
			\draw[black,thick](6,1) node (c2)  {$a^*_{2}$};
			\draw[black,thick](8,1) node (c3)  {$a^*$};
			\draw[black,thick](2,-3) node (d)  {$b_{1}$};
			\draw[black,thick](4,-3) node (d1)  {$b_{2}$};
			\draw[black,thick](6,-3) node (d2)  {$b^*_{2}$};
			\draw[black,thick](8,-3) node (d3)  {$b^*_{1}$};
			
		}
		\path[draw,thick,black!60!white,dashed]
		(c) edge node[name=la,pos=0.7, above] {\color{blue} } (d)
		(c1) edge node[name=la,pos=0.7, above] {\color{blue} } (d1)	
		(c3) edge node[name=la,pos=0.7, above] {\color{blue} } (d3)	
		(c2) edge node[name=la,pos=0.7, above] {\color{blue} } (d2)	
		;
		
		\path[draw,thick,black]
		(d) edge node[name=la,pos=0.7, above] {\color{blue} } (d1)
		(c2) edge node[name=la,pos=0.7, above] {\color{blue} } (c1)
		(d3) edge node[name=la,pos=0.7, above] {\color{blue} } (d2)
		;
		\node at (5,-4.8) {$(c)$};	
		\end{scope}	
		
		\end{tikzpicture}
			  	\end{center}
	\caption{Alternating paths in Step 2 of our coloring phase. Dashed lines indicate uncoloured edges, and solid
		lines indicate edges with color $i$.}
	\label{f1}
\end{figure}
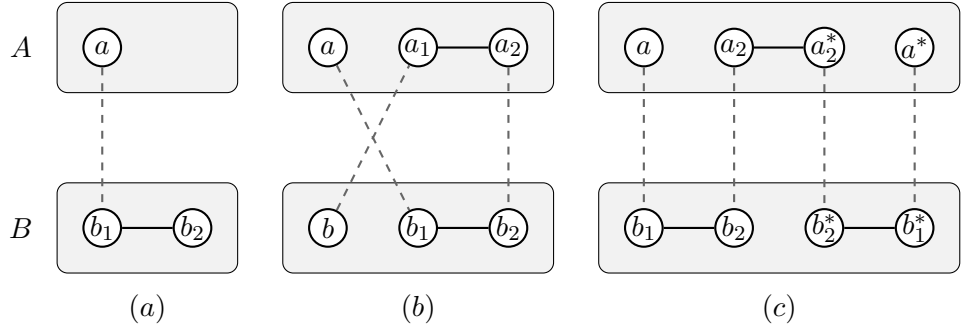

Suppose first  that we have some MCC-pair $(a,b)$ with respect to $i$, with $a\in A$ and $b\in B$. Note that by our procedure aboce, we know that $a,b\not\in W$. By~\eqref{eqn:N-M-sizea}, we have 
$
|N_2(a)|,  |N_2(b)|>\tfrac{1}{4}(1+0.5\ve) n.
$
We choose $a_1a_2$ with color $i$ such that $a_1\in N_1(b)$
and $a_2\in N_2(b)$. Now as $|N_1(a)|, |N_1(a_2)| >\tfrac{1}{4}(1+0.5\ve) n$ by~\eqref{eqn:N-M-sizea}, we know that $|N_1(a_2)\cap N_2(a) |>\frac{1}{4}\ve n$. We choose $b_2\in N_1(a_2)\cap N_2(a)$
so that  $b_2a_2 \not\in M$, 
and let $b_1\in N_1(a)$ such that $b_1b_2$ is colored by $i$. 
Then $P=ab_1b_2a_2a_1 b$ is an alternating path from $a$ to $b$ (See Figure~\ref{f1}(b)). We exchange $P$ by coloring $ab_1, b_2a_2$ and $a_1b$
with color $i$ and uncoloring the edges $a_1a_2$ and $b_1b_2$. 
After the exchange,  the color $i$ appears on edges incident with $a$ and $b$,
the edge $a_1a_2$ is added to $R_A$
and the edge $b_1b_2$ is added to $R_B$.

Suppose now that we have an MCC-pair $(a,a^*)$
with respect to  $i$ such that $a,a^* \in A\setminus W$ (the case for an MCC-pair $(b,b^*)$
with $b,b^* \in B$ can be handled similarly). 
By~\eqref{eqn:N-M-sizea}, we have  $|N_2(a^*)|>\tfrac{1}{4}(1+0.5\ve) n$.
We take an edge $b_1^*b_2^*$ colored by $i$ with $b_1^* \in N_1(a^*)$ and  $b_2^*\in N_2(a^*)$.  
Then again, by~\eqref{eqn:N-M-sizea}, we have 
$
|N_2(a)|,  |N_2(b_2^*)|>\tfrac{1}{4}(1+0.5\ve) n.  
$
Therefore, as each vertex  $c\in N_2(b_2^*)$ 
satisfies $|N_1(c)|>\tfrac{1}{4}(1+0.5\ve) n$,  we have $|N_1(c)\cap N_2(a)| >\tfrac{1}{4}\ve n$. We take $a_2a_2^*$ colored by $i$ with $ a_2^*\in N_1(b_2^*)$ such that $a_2^*b_2^*\not\in M$ and  $a_2\in   N_2(b_2^*)$.   
Then we let $b_2\in N_1(a_2)\cap N_2(a)$ such that $a_2b_2\not\in M$, and let $b_1$
be the vertex in $N_1(a)$ such that $b_1b_2$ is colored by $i$. 
Now we get the alternating path $P=ab_1b_2 a_2 a_2^* b_2^* b_1^* a^*$ (See Figure~\ref{f1}(c)).  
We exchange $P$ by coloring $ab_1, b_2a_2, a_2^*b_2^*$ and $b_1^*a^*$
with color $i$ and uncoloring the edges $b_1b_2, b_1^*b_2^*$ and $a_2a_2^*$. 
After the exchange,  the color $i$ appears on edges incident with $a$ and $a^*$,
the edges $b_1b_2$ and $b_1^*b_2^*$ are added to $R_B$
and the edge $a_2a_2^*$ is added to $R_A$.  In this process, 
we  added  one edge  to $R_A$
and at most two edges to $R_B$.

We have shown above how to find alternating paths between all MCC-pairs, and to modify $\varphi$ by switching on these.  Each time we make such a switch, we increase by one the size of the color class $i$ (and decrease by 2 the number of vertices that are missing color $i$). By repeating this process, we can therefore continue until the color class $i$ is a perfect matching of $Q$ in Case 1, and saturates $V(Q)\setminus (U\setminus W)$ in case 2. During this whole process, we  kept exactly  $k$ edges of $M\cup E(x)$ to be colored all differently, and kept the rest edges of 
$M\cup E(x)$ to be uncolored. 
 Hence we still have that all edges in $M_1$ receive different colors and all edges in $(M \cup E(x))\setminus M_1$ are uncolored. \\

\noindent \textbf{Step 3: Coloring $R_A$ and $R_B$ and extending the new color classes.}

Let  $p= |M\cap E(R_A \cup R_B)|$ and  $\ell= 2r+p$. 
In this step we leave the colors $1, 2, \ldots, k$ alone, but extend $\varphi$ by introducing $\ell$ new colors, say $k+1, \ldots, k+\ell$. 

Note that $p=0$ unless $n$ is odd (the edges of $M\cap E(R_A \cup R_B)$ are some   $M$-edges  that got uncolored when we dealt with missing colors at $v_1=x$ when $n$ is odd).  Let us first deal with the colors $k+1, \ldots, k+p$. Here, we start by coloring all the edges of $E(R_A \cup R_B) \cap M$ with $p$ distinct colors.  
If we look at the missing color count in~\eqref{eqn:total-number-missing-colors-Cases1-2} and~\eqref{eqn:total-number-missing-colors-Case3-1}, we see that we had at least $2n$ MCC-pairs in Step 2, each of which corresponds to at least one edge in each of $Q[A], Q[B]$ that will be uncolored at the end of Step 2. So we know that $e(R_A), e(R_B) \ge 2n > 2p$. For each color $i\in [k+1, k+p]$, 
and each $e_i=uv\in E(R_A \cup R_B)\cap M$, say $e_i \in E(R_A)$, we let $f=yz \in E(R_B)\setminus M$ such that $f$ is uncolored so far. 
We color both $e$ and $f$ by $i$.  
Similarly, if  $e_i \in E(R_B)$, we let $f=yz \in E(R_A)\setminus M$ such that $f$ is uncolored so far. 
We color both $e$ and $f$ by $i$.   
We define $A_i, B_i$ as the sets of vertices in $A, B$, respectively, that are incident with edges colored $i$. 
Then we have  $|A_i|=|B_i| =2$.  
Let $H_i$ be the subgraph of $Q[A, B]-M$ obtained by
deleting the vertex sets $A_i \cup B_i$ and removing all colored edges. 
Then from~\eqref{eqn:N-M-sizea}, \eqref{eqn:size-of-N_1(v_0)},  and~\eqref{eqn:size-of-N_1(v_0)0}, 
for any $v\in V(H_i)\setminus W$, we have 
\begin{eqnarray}
d_{H_i}(v)  &\ge& \tfrac{1}{4}(1+0.5\ve) n -p \ge \tfrac{1}{4}(1+0.4\ve) n,  \nonumber  
\end{eqnarray}
and for any  $v\in W$, we have 
\begin{eqnarray} 
d_{H_i}(v)  & \ge&   0.1\ve n -p >0.09\ve n.\nonumber 
\end{eqnarray}
Since $|W|< \xi n <0.09 \ve n$, Lemma~\ref{lem:matching-in-bipartite} implies that 
$H_i$ has a perfect matching $M_i$. We color all the edges of $M_i$
by color $i$. 

We now proceed to deal with the $\ell-p$ colors $k+p+1, \ldots, k+\ell$. Let us still call by $R_A$ and $R_B$ the remaining graphs in $Q[A]$ and $Q[B]$ respectively with uncolored edge at this point (i.e. removing those edges colored in the last paragraph); as more edges are colored in this step we continue to update $R_A,R_B$.  By (C2), we have  still have $\Delta(R_A), \Delta(R_B)<r$. 
In Case 1, as the only multiple edges of $Q$ were added  through the  the vertex set $B_0$, and $v_1\not\in V(R_B)$, we have 
 $\mu (R_A\cup R_B) \le  \mu(Q) \le 2$. In Case 2, all the possible multiple edges of $Q^{**}$ 
 were added through the vertex set $B_0$ and within $U^*$, and we have  
 $\mu (R_A\cup R_B) \le  \Delta(R_A\cup R_B)<r$. 
Theorem~\ref{thm:chromatic-index} gives us an $(\ell-p)$-edge-coloring of $R_A\cup R_B$ using the colors $k+p+1, \ldots, k+\ell$.  By Theorem~\ref{lem:equa-edge-coloring}, we may assume that all these new color classes differ in size by at most one. In addition, in Case 1 we have $e(R_A)=e(R_B)$, so by possibly renaming some color classes we may assume that each color appears on the same number of edges in $R_A$ as it does in $R_B$. We now work on extending our $\ell-p$ new color classes.

By (C1) we had $e(R_A), e(R_B)< s$, so since  $\ell>r$, each of
the colors $k+p+1,\ldots, k+\ell$ appears on fewer than $s/r+1$
edges in each of $R_A$ and $R_B$.  Note that we have 
\begin{numcases}{s/r+1 \le }
    3 \xi^{1/2} n+1  & Case 1, \nonumber  \\
    3.5 n^{5/6}+1, & Case 2.  \nonumber 
\end{numcases}
We will now color some of the edges of $Q[A, B]- M$ with the $\ell-p$ new colors so that in Case 1, each of these color classes is a perfect matching, and in Case 2, each of the new colors is seen by all the vertices of $V(Q)\setminus U$.

Let $i$ be a color in $[k+p+1,k+\ell]$. We define $A_i, B_i$ as the sets of vertices in $A, B$, respectively, that are incident with edges colored $i$. 
Then, from our discussion in the previous paragraph, $|A_i|,|B_i| < 2 (s/r+1)$. 
In Case 1, 
let $H_i$ be the subgraph of $Q[A, B]-M$ obtained by
deleting the vertex sets $A_i \cup B_i$ and removing all colored edges. As $|A_i|=|B_i|$ in this case, we know that $H_i$
is a balanced bipartite graph. 
In Case 2, when $n$ is even, we  choose an uncolored edge from $E(x)$, say $xw$,
where $w$ can be in $A$ or $B$, such that $w\not\in A_i\cup B_i$ (this is possible as $|E(x) \setminus M_1| \ge \ve n+2> |A_i\cup B_i| +\ell$). We then color $xw$ by $i$,
and include $w$ in either $A_i$ (if $w\in A$) or $B_i$ (if $w\in B$).  We still call the resulting set as $A_i$ or $B_i$.  We are able to choose edges from $E(x)$ 
in this entire step, as 
if $|M| \ge k$, 
we know that $|E(x)\setminus M_1| =|E(x)| =\Delta+1-(n-\Delta-1) \ge \ve n+2 >\ell$;  
and if $|M|<k$, then   
 by our choice of $M_1$ (that selected edges from $M$ with priority), we know that 
$
    |E(x)\setminus M_1| \ge  (\Delta+1-|M|) -(k-|M|) \ge \frac{1}{4}n-1>\ell. 
$
In Case 2, when $n$ is odd, we  choose an uncolored edge, say $xw (=v_1w)$ from $E(x)$ and perform the same operation as above. 

We will show that $H_i$ has a  matching of desired size and we will color these matching edges with $i$.
Then from~\eqref{eqn:N-M-sizea}, \eqref{eqn:size-of-N_1(v_0)},  and~\eqref{eqn:size-of-N_1(v_0)0},
for any $v\in V(H_i)\setminus W$, we have 
\begin{eqnarray}
d_{H_i}(v)  & \ge& \tfrac{1}{4}(1+0.5\ve) n -\ell \nonumber  \\
 &\ge & \tfrac{1}{4}(1+0.4\ve) n,  \nonumber  
\end{eqnarray}
and for any  $v\in W$, we have 
\begin{eqnarray} 
d_{H_i}(v)  & \ge&   0.1\ve n -\ell >0.09\ve n.\nonumber 
\end{eqnarray}

In Case 1,  
since $|W|< \xi n <0.09 \ve n$, Lemma~\ref{lem:matching-in-bipartite} implies that 
$H_i$ has a perfect matching $M_i$. We color all the edges of $M_i$
by color $i$.

In Case 2,  for each $i\in [k+1, \ell]$, assume, without loss of generality, that $|A_i| \ge |B_i|$. 
We let $C_i\subseteq (U\cap B)\setminus B_i$ such that  $|C_i| =|A_i|-|B_i|$. Thus is possible as $|U\cap B| \ge \lfloor 0.4\xi n \rfloor >2(s/r+1) >|A_i|$.  
We find in $H_i-C_i$ a 
 perfect matching $M_i$ by Lemma~\ref{lem:matching-in-bipartite}. 
Then we color the edges of $M_i$ with the color $i$. 

We denote the set of edges from $M\cup E(x)$ colored in Step 3 by 
$M_2$. \\

\noindent \textbf{Step 4: Coloring the remaining graph $R$ and getting a good coloring of $G^M$.}

Let $R$ be the graph induced by the remaining uncolored edges of $Q$ at this point, as well as the edges of $E(x) \setminus (M_1\cup M_2)$. Note that the graph $R-x$ is bipartite (and 
when $n$ is odd, $R$ is bipartite). Note also that $R$ is simple. 

Our goal is to get a good coloring of $G^M$, that is, a $(\Delta+2)$-edge-coloring of $G^M$ where all the edges of $M\cup E(x)$ receive different colors. Towards this goal, all edges of $E(G^M)\setminus E(R)$ have already been colored, using $k+\ell$ colors, in such a way that all edges of $M\cup E(x)$ that have been colored have received different colors. In this final step we will introduce $\Delta+2-k-\ell$ new colors to complete the job.

Suppose first that we are in case 1. Consider $d_R(v)$ for some $v\in V(G)$. Note that 
$v$ has degree exactly one larger in $G^M$ than in $G$, due to its incident edge in $M\cup E(x)$, and it may have one additional degree in $Q$ due to the additional edges we added in construction step 2 (i.e. if $v\in B_0$ there). Since we are in case 1, $v$ is already incident to exactly $k+\ell$ different colored edges, so we get that
\begin{numcases}{d_R(v)\leq }
    \Delta+2-k-\ell \hspace*{.1in}\textrm{ if $v\in B\setminus \{x\}$,}  \nonumber\\
    \Delta+1-k-\ell \hspace*{.1in}\textrm{ if $v\in A$}. \nonumber 
\end{numcases}
In this case,  as all the edges of $Q[B]$ were colored already, we have 
$$d_R(x) \le n-2|M| \le 0.2 \ve n+1\leq \Delta+2-k-\ell.$$

 
Suppose now that we are in case 2.  If $|M| =\lfloor \frac{n}{2} -0.1\ve n\rfloor$,
then we get $d_R(x) \le \Delta+2-k-\ell$ the same way as above. Thus we assume that $|M|=n-1-\Delta$. 
In this case, by Step 3,  we know that $|M_1\cup M_2|=k+\ell$. As all edges of $Q[B]$ were colored already,  we get
$$ d_R(x) =|E(x) \setminus (M_1\cup M_2)| \le |M\cup E(x)| -k-\ell = \Delta+1-k-\ell.$$ 
If $v\in V(G)\setminus (U\setminus W)$, then $v$ is incident to exactly $k+\ell$ colored edges already, so the same argument from the above paragraph tells us that
\begin{numcases}{d_R(v)\leq}
    \Delta+2-k-\ell \hspace*{.1in}\textrm{ if $v\in B\setminus (U\setminus W)$}, \nonumber \\
    \Delta+1-k-\ell \hspace*{.1in}\textrm{ if $v\in A\setminus (U\setminus W)$}. \nonumber 
\end{numcases}
For any $v\in U\setminus W$, we have 
$d_R(v) \le \Delta+2-(k-|\pbar_0(v)|),$ where $|\pbar_0(v)|$ is the number of colors missing at $v$ after coloring edges in $Q_{AB}$ in Step 1. 
As $|\pbar_0(v)| <3.4 n^{2/3}$ by~\eqref{eqn:number-missing-color-Case2} and the definition of $W$, we have $d_R(v) \le \Delta -0.4\xi n - \frac{1}{2}\Delta -n^{2/3} -3+3.4 n^{2/3} +1<\Delta(R)$, where we have $d_{Q^{**}}(v) \le \Delta -0.4\xi n$ by Claim~\ref{claim:vertex-degree-in-Q^*}. 

We now want to apply Lemma 13(b) to the graph $R$ with bipartition $(A, B\setminus \{x\})$ of $R-x$, $J=J_0:=(M\cup E(x))\setminus (M_1\cup M_2)$, $\Delta(R[J]-x) \le 1$,  and the number of colors $c:=\Delta+2-k-\ell$. Note that $R$ is a simple graph with $\Delta(R)\leq c$, and note that every vertex in $A$ has degree strictly less than $c$. Since $|J|\leq  \Delta(R)\leq c$, we can 
color each edge of $J$ using a distinct color from our palette of $c$ new colors to get a pre-coloring. But then the lemma guarantees a $(\Delta+2-k-\ell)$-edge-coloring  of $R$ 
such that each edge of $J$ is colored differently.  This completes our proof.\qed 

\bibliographystyle{abbrv}
\bibliography{BIB}

\begin{thebibliography}{10}

\bibitem{Behzad-total-coloring}
M.~Behzad.
\newblock {G}raphs and {T}heir {C}hromatic {N}umbers.
\newblock {\em Ph.D. thesis, Michigan State University, East Lansing}, 1965.

\bibitem{BergeTBformula}
C.~Berge.
\newblock Sur le couplage maximum d'un graphe.
\newblock {\em C. R. Acad. Sci. Paris}, 247:258--259, 1958.

\bibitem{MR4694336}
Y.~Cao, G.~Chen, G.~Jing, and S.~Shan.
\newblock The core conjecture of {H}ilton and {Z}hao.
\newblock {\em J. Combin. Theory Ser. B}, 166:154--182, 2024.

\bibitem{DMS}
A.~Dalal, J.~McDonald, and S.~Shan.
\newblock Total coloring graphs with large maximum degree.
\newblock {\em Accepted by Journal of Graph Theory}, 2025.

\bibitem{Dirac}
G.~A. Dirac.
\newblock Some theorems on abstract graphs.
\newblock {\em Proc. London Math. Soc. (3)}, 2:69--81, 1952.

\bibitem{MR4679850}
J.~Geetha, N.~Narayanan, and K.~Somasundaram.
\newblock Total colorings---a survey.
\newblock {\em AKCE Int. J. Graphs Comb.}, 20(3):339--351, 2023.

\bibitem{MR2028248}
S.~Gr\"unewald and E.~Steffen.
\newblock Independent sets and 2-factors in edge-chromatic-critical graphs.
\newblock {\em J. Graph Theory}, 45(2):113--118, 2004.

\bibitem{Gupta-67}
R.~G. Gupta.
\newblock {\em Studies in the Theory of Graphs}.
\newblock PhD thesis, Tata Institute of Fundamental Research, Bombay, 1967.

\bibitem{HS}
A.~Hajnal and E.~Szemer\'edi.
\newblock Proof of a conjecture of {P}. {E}rd{\H{o}}s.
\newblock In {\em Combinatorial theory and its applications, {I}-{III} ({P}roc. {C}olloq., {B}alatonf\"ured, 1969)}, volume~4 of {\em Colloq. Math. Soc. J\'anos Bolyai}, pages 601--623. North-Holland, Amsterdam-London, 1970.

\bibitem{MR0297607}
A.~Hajnal and E.~Szemer\'{e}di.
\newblock {P}roof of a conjecture of {P}. {E}rd{\H{o}}s.
\newblock In {\em Combinatorial theory and its applications, {I}-{III} ({P}roc. {C}olloq., {B}alatonf\"{u}red, 1969)}, volume~4 of {\em Colloq. Math. Soc. J\'{a}nos Bolyai}, pages 601--623. North-Holland, Amsterdam-London, 1970.

\bibitem{H1962}
S.~L. Hakimi.
\newblock On realizability of a set of integers as degrees of the vertices of a linear graph. {I}.
\newblock {\em J. Soc. Indust. Appl. Math.}, 10:496--506, 1962.

\bibitem{MR1226136}
A.~J.~W. Hilton and H.~R. Hind.
\newblock The total chromatic number of graphs having large maximum degree.
\newblock {\em Discrete Math.}, 117(1-3):127--140, 1993.

\bibitem{Holyer}
I.~Holyer.
\newblock The np-completeness of edge-coloring.
\newblock {\em SIAM Journal on Computing}, 10(4):718--720, 1981.

\bibitem{JT}
T.~R. Jensen and B.~Toft.
\newblock {\em Graph coloring problems}.
\newblock Wiley-Interscience Series in Discrete Mathematics and Optimization. John Wiley \& Sons, Inc., New York, 1995.
\newblock A Wiley-Interscience Publication.

\bibitem{MR1511872}
D.~K\"{o}nig.
\newblock \"{U}ber {G}raphen und ihre {A}nwendung auf {D}eterminantentheorie und {M}engenlehre.
\newblock {\em Math. Ann.}, 77(4):453--465, 1916.

\bibitem{MR300623}
C.~J.~H. McDiarmid.
\newblock The solution of a timetabling problem.
\newblock {\em J. Inst. Math. Appl.}, 9:23--34, 1972.

\bibitem{PS2023}
M.~J. Plantholt and S.~Shan.
\newblock Edge coloring graphs with large minimum degree.
\newblock {\em J. Graph Theory}, 102(4):611--632, 2023.

\bibitem{Plantholt-Shan}
M.~J. Plantholt and S.~Shan.
\newblock On the multigraph overfull conjecture.
\newblock {\em Journal of Graph Theory}, 109(2), 2025.

\bibitem{SHAN2022429}
S.~Shan.
\newblock Chromatic index of dense quasirandom graphs.
\newblock {\em J. Combin. Theory Ser. B}, 157:429--450, 2022.

\bibitem{StiebSTF-Book}
M.~Stiebitz, D.~Scheide, B.~Toft, and L.~M. Favrholdt.
\newblock {\em Graph edge coloring}.
\newblock Wiley Series in Discrete Mathematics and Optimization. John Wiley \& Sons, Inc., Hoboken, NJ, 2012.
\newblock Vizing's theorem and Goldberg's conjecture, With a preface by Stiebitz and Toft.

\bibitem{NPhTotal}
A.~Sánchez-Arroyo.
\newblock Determining the total colouring number is np-hard.
\newblock {\em Discrete Mathematics}, 78(3):315--319, 1989.

\bibitem{TutteTBformula}
W.~T. Tutte.
\newblock The factorization of linear graphs.
\newblock {\em Journal of the London Mathematical Society}, s1-22(2):107--111, 1947.

\bibitem{Vizing-2-classes}
V.~G. Vizing.
\newblock Critical graphs with given chromatic class.
\newblock {\em Diskret. Analiz No.}, 5:9--17, 1965.

\bibitem{Vizing-total-coloring}
V.~G. Vizing.
\newblock Some unsolved problems in graph theory.
\newblock {\em Russ. Math. Surv.}, 23:125--141, 1968.

\bibitem{WestText}
D.~B. West.
\newblock {\em Introduction to Graph Theory}.
\newblock Prentice Hall, 2 edition, September 2000.

\end{thebibliography}
\end{document}